\documentclass[smallextended, 
envcountsect]{svjour3} 
\smartqed 
\usepackage{graphicx}

\usepackage{mathtools}
\mathtoolsset{showonlyrefs,showmanualtags}

\usepackage{amsmath,amssymb}
\usepackage{mathrsfs}
\usepackage{enumerate}
\usepackage{latexsym,amsfonts}
\usepackage{verbatim}
\usepackage{color}
\usepackage{pb-diagram}
\usepackage[normalem]{ulem}

\usepackage{tikz}

\usepackage{marginnote}

\definecolor{green6}{rgb}{0.0,0.6,0}
 \definecolor{purple6}{rgb}{0.6,0,0.4}

 \definecolor{blue8}{rgb}{0,0.36,0.6}
 \definecolor{blue6}{rgb}{0.0,0.27,0.45}
 \definecolor{blue4}{rgb}{0.0,0.18,0.3}
 \definecolor{blue2}{rgb}{0.0,0.0,0.2}

 \definecolor{red8}{rgb}{0.6,0,0.36}
 \definecolor{red6}{rgb}{0.0,0.27,0.45}
 \definecolor{red4}{rgb}{0.0,0.18,0.3}
 \definecolor{red2}{rgb}{0.0,0.0,0.2}

\newcommand{\R}{\mathbb{R}}
\newcommand{\qo}{\text{ a.e.~}}
\newcommand{\wh}{\widehat}

\newcommand{\wxo}{\wh x_0}
\newcommand{\wpsi}[1]{\wh\psi_{#1}}
\newcommand{\vecc}[1]{\overrightarrow{#1}}

\newcommand{\sss}{\wh s}
\newcommand{\sg}{\mathfrak{s}}

\newcommand{\beps}{\boldsymbol{\varepsilon}}

\newcommand{\bk}{\boldsymbol{\theta}}

\newcommand{\wI}{\widehat{I}}
\newcommand{\wbx}{\widehat{\boldsymbol{x}}}
\newcommand{\wxi}{\widehat{\xi}}

\newcommand{\wu}{\widehat{u}}

\newcommand{\wla}{\widehat{\lambda}}
\newcommand{\wtau}{\widehat{\tau}}
\newcommand{\well}{\widehat{\ell}}

\newcommand{\cH}{{\mathcal H}} 
\newcommand{\cK}{{\mathcal K}} 
\newcommand{\cHref}{\wh{\mathcal H}} 
\newcommand{\vF}[1]{\overrightarrow{F_{#1}}} %
\newcommand{\vGse}[1]{\overrightarrow{G\se_{#1}}} %
\newcommand{\vH}[1]{\overrightarrow{H_{#1}}} %
\newcommand{\vPsi}{\overrightarrow{\Psi}} %

\newcommand{\se} {^{\prime \prime}} %
\newcommand{\dl}{{\delta\ell}} %
\newcommand{\de}{{\delta e}} %
\newcommand{\dep}{\delta p} %
\newcommand{\dx}{{\delta x}} %
\newcommand{\dy}{\delta y} %

\newcommand{\lref}{\widehat\ell}
\newcommand{\lo}{{\lref_0}} %

\newcommand{\wS}[1]{\widehat S_{#1}} %
\newcommand{\wSinv}[1]{\widehat S^{-1}_{#1}} %

\newcommand{\liebr}[2]{\left[  {#1}, {#2} \right] } %
\newcommand{\liede}[3]{ L_{#1}{#2} \left( {#3}\right) } %
\newcommand{\lieder}[2]{L_{#1}{#2}(\wbx_0)} %
\newcommand{\liederder}[3]{ L_{#1}L_{#2}{#3}(\wbx_0)  } %
\newcommand{\liederdo}[2]{ L^2_{#1}{#2} \left( \wbx_0 \right) } %

\newcommand{\abs}[1]{\left\vert {#1} \right\vert} %
\newcommand{\scal}[2]{\langle {#1} \, , \, {#2} \rangle} %
\newcommand{\dueforma}[2]{{\boldsymbol\sigma}\left( {#1}, {#2} \right) } %

\newcommand{\ep}{\varepsilon}

\newcommand{\cU}{{\mathcal U}} %

\newtheorem{hypo}{Assumption}

\usepackage[normalem]{ulem}
\begin{document}

\title{Strong Local Optimality for Generalized L1 Optimal Control Problems}

\author{Francesca C. Chittaro   \and  Laura Poggiolini }

\institute{Francesca C. Chittaro \at
            Aix Marseille Univ, Universit\'e de Toulon, CNRS, LIS, Marseille, France 
              francesca.chittaro@univ-tln.fr  
           \and
           Laura Poggiolini,  Corresponding author  \at
           Dipartimento di Matematica e Informatica ``Ulisse Dini'', 
           Universit\`a degli Studi di Firenze,\\
           via di Santa Marta  3,  50139 Firenze, Italy
           laura.poggiolini@unifi.it
}

\date{Received: date / Accepted: date
}

\maketitle

\begin{abstract}

In this paper, we analyze control affine optimal control problems with a cost functional involving the absolute value of the control.
The Pontryagin extremals associated with such systems are given by (possible) concatenations of bang arcs with singular arcs and with 
\emph{zero control arcs}, that is, arcs where the control is identically zero.
Here, we consider Pontryagin extremals given by a  bang-zero control-bang concatenation. 
We establish sufficient optimality conditions for such extremals, in terms of some regularity conditions and of the coerciveness of a 
suitable finite-dimensional second variation. 
\end{abstract}
\keywords{Optimal Control, sufficient conditions, Hamiltonian methods, sparse control}
\subclass{49K15 49K30}


\section{Introduction}

In recent years, optimal control problems aiming at minimizing the integral of the absolute value of the control 
have received an increasing attention: 
for instance, they model problems 
coming from neurobiology \cite{PLOS}, mechanics \cite{yacine,maurer-vossen}, and fuel-consumption \cite{boizot,craig,ross}.
As noticed in these papers, a peculiarity of 
optimal control problems containing the absolute value of the control in the running cost 
is the fact that the optimal control vanishes along nontrivial time intervals; this property  is referred to as \emph{sparsity}, 
and the piece of optimal trajectories corresponding to the zero control  are known as \emph{zero control arcs}, \emph {cost arcs} or \emph{inactivated  arcs}. 

The occurrence  of sparse controls in this kind of minimization problems is well known also in the framework of infinite-dimensional optimal control, 
see for instance \cite{clason-kunisch,kunisch-sparse,stadler}, where usually the absolute value of the control 
is added to the integral cost, in order to induce sparse solutions.

In this paper, we consider optimal control problems with single-input con\-trol-affine dynamics, a compact control set and a  running cost depending  linearly on the absolute value of the control. 
For such optimal control problems, a well known necessary condition is the Pontryagin Maximum Principle, possibly in its non-smooth version; see \cite{Clarke}.
Our aim is to provide sufficient conditions for the strong local optimality of an admissible trajectory-control pair satisfying the Pontryagin Maximum Principle.
A first announcement of these results, given without proofs, can be found in \cite{PoggChittCDC}.

Our approach relies on Hamiltonian methods, described in Section \ref{sec:Ham}; see also \cite{AgSac,SZ16,ASZ02,PoggSte04}.
In order to apply these methods, two main ingredients are needed. Namely, we require some regularity conditions along the reference extremal, and the coerciveness of the so-called \emph{second variation at the switching points}.  

We point out that these conditions are quite easy to check on a given extremal: indeed, the regularity assumptions are just sign 
conditions along the extremal,
while the second variation reduces to a quadratic function of one real variable. The assumptions on the zeros of the cost along the reference trajectory 
(Assumptions~\ref{NT}-\ref{SS}) are generic.

The cited papers \cite{ASZ02,PoggSte04} deal respectively with the Mayer problem and the minimum time problem, in a completely smooth framework. Therefore, the conditions found in  there cannot be directly applied to this context; in particular, the second variation
is rather different in the three cases. 

The Hamiltonian approach is used also in \cite{yacine}, where a similar problem is studied, and sufficient conditions for optimality are provided. In particular,  
the same regularity assumptions are required in the two results (we stress indeed the fact that the regularity conditions are very close to the necessary ones, 
in the sense that they sum up to requiring strict inequalities 
where the Pontryagin Maximum Principle yields mild ones). The condition that ensures the invertibility of the flow 
(in our context, the coerciveness of the second variations), in \cite{yacine} has the form of a non-degeneration condition on the Jacobian of the exponential mapping.

The two results thus present a set of alternative sufficient conditions, that can adapt to different cases. 

We finally mention the paper \cite{maurer-vossen}, where the authors investigate necessary and sufficient optimality conditions for a similar problem, with different techinques. 
The given optimal control problem is extended in such a way that the candidate optimal trajectory happens
to be a bang-bang extremal of the extended problem, so that sufficient conditions are stated in terms of optimality conditions of finite-dimensional optimization problems.

\medskip
This paper is organized as follows: in Section~\ref{sec: preliminaries} we state the problem and discuss the main Assumptions; in Section~\ref{sec:Ham} we describe the Hamiltonian methods; the main result is given in Section~\ref{sec: main}. 
Finally, in Section~\ref{sec:example} we provide an example.

\section{Statement of the Problem and Assumptions} \label{sec: preliminaries}

\subsection{Notations}
Let $M$ be a smooth $n$-dimensional manifold.
Given a vector field $f$ on $M$, the  Lie derivative at a point $q\in M$ of  a smooth function $\varphi \colon M \to \mathbb{R}$ with respect to $f$ is denoted with
$L_f \varphi(q)= \langle d\varphi(q),f(q)\rangle$. We  denote the iterated Lie derivative  $L_f\big(L_f\varphi\big)(q)$ as $L_f^2\varphi(q)$.

The Lie bracket of two vector fields $f,g$ is denoted as commonly with $[f,g]$.

Let $T^*M$ be the cotangent bundle of $M$ and $\pi\colon T^*M \to M$ its projection on $M$.
For every vector field $f$ on $M$, denote by the corresponding capital letter the associated Hamiltonian function, that is 
$F(\ell)=\langle \ell,f(\pi \ell) \rangle$, $\ell \in T^*M$.

Let $\boldsymbol{\sigma}$ be the canonical symplectic form on $T^*M$. 
With each Hamiltonian function $F$ we associate the Hamiltonian vector field $\vecc{F}$ on $T^*M$
defined by 
\[
\langle d F(\ell), \cdot \rangle = \boldsymbol{\sigma}(\cdot,\vecc{F}(\ell)).
\]
The Poisson bracket of two Hamiltonians $F,G$ is denoted $\{F,G\}$. We recall that $\{F,G\}=\boldsymbol{\sigma}(\vecc{F},\vecc{G})$.

The Hamiltonian flow associated with the Hamiltonian vector field $\vecc{F}$, and starting at time $t = 0$, is denoted with the cursive $\mathcal{F}_t$.

\subsection{Statement of the Problem}

We consider the following minimization  problem on $M$: 
\begin{subequations} \label{problema}
\begin{align}
& \text{minimize } \int_0^T \abs{u(t)\psi(\xi(t))} dt \text{ subject to} \label{eq: costo} \\
& \dot\xi(t) = f_0 (\xi(t)) + u(t) f_1(\xi(t)), \label{eq: contr sys} \\
& \xi(0)= \wbx_0, \quad \xi(t) = \wbx_f, \label{eq: vincoli} \\
& \abs{u(t)} \leq 1 \quad \qo t \in [0,T] \label{eq: controllo},
\end{align}
\end{subequations}
where $f_0$ and $f_1$ are smooth vector fields on $M$,  $\psi$ is a smooth real-valued function on $M$ and $T>0$  is fixed. This problem is indeed a generalization of the $L^1$ minimization problem considered, for instance, in \cite{maurer-vossen}.

The Pontryagin Maximum Principle (PMP) constitues a necessary condition for optimality.
In particular, it states that for every optimal trajectory $\xi(\cdot)$ of \eqref{problema}
there exist a constant $\nu\leq 0$ and a suitable curve $\lambda (t)\in T_{\xi(t)}^*M$ such that 
\begin{equation}\label{eq:maxcond}
u(t)F_1(\lambda(t)) +\nu |u(t) \psi(\xi(t))| =\max_{v\in[-1,1]} 
\left\{ vF_1(\lambda(t)) +\nu |v \psi(\xi(t))| \right\},  \qo t.
\end{equation}
In particular, in the normal case ($\nu=-1$), equation  \eqref{eq:maxcond} implies that the optimal control may have three possible behaviors:
\begin{itemize}
 \item If there exists an interval $I\!\subset\![0,T]$ such that $\abs{F_1(\lambda(t))} > \abs{\psi(\xi(t))}\quad \forall t \in I$, then  
 $u(t) = \mathrm{sgn}(F_1(\lambda(t))$. In this case we say that the interval $I$ is a {\it regular bang arc}, and we call $u$ a \emph{bang control}. 
\item If $\abs{F_1(\lambda(t)))} < \abs{\psi(\xi(t))} \quad \forall t \in I$, then the Hamiltonian \eqref{eq:maxcond} attains its maximum  
only if $u\equiv 0$ in $I$. 
To denote such arcs we use the descriptive name of {\it zero control arcs}.  
We recall, see \cite{PLOS}, that in the literature these arcs are also known as  {\it inactivated arcs}.
\item If $\abs{F_1(\lambda(t))} = \abs{\psi(\xi(t))}$ $\;\forall t \in I$, then the maximizing control is not uniquely determined by \eqref{eq:maxcond}. Indeed, if $F_1(\lambda(t)) = \abs{\psi(\xi(t))}$, then the 
 maximum in \eqref{eq:maxcond} is attained by every  $v\in[0,1]$. Analogously, if  $F_1(\lambda(t)) = - \abs{\psi(\xi(t))}$, then the maximum  is attained by every  $v \in [-1,0]$. We say that $I$ is a {\it singular arc}.
\end{itemize}
In this paper we state sufficient optimality conditions for solutions of \eqref{eq: contr sys}-\eqref{eq: vincoli}-\eqref{eq: controllo} 
given by a concatenation of bang and 
zero control arcs.
More precisely, we assume that we are given an admissible trajectory  $\wxi \colon [0,T]\to M$ satisfying 
\begin{align*}
& \dot{\wxi}(t)=
\begin{cases}
f_0(\wxi(t))+u_1 f_1(\wxi(t)) & t\in[0,\wtau_1)\\ 
f_0(\wxi(t)) & t\in(\wtau_1,\wtau_2)\\
f_0(\wxi(t))+u_3 f_1(\wxi(t)) & t\in(\wtau_2,T]\\
\end{cases}  \qquad
\wxi(0) = \wbx_0, \quad \wxi(T) = \wbx_f,
\end{align*}
for some $\wtau_1$, $\wtau_2 \in (0,T)$, $\wtau_1 < \wtau_2$ and some $u_1$, $u_3 \in \{-1, 1 \}$. The pair $(\wxi,\wu)$, where 
\begin{equation} \label{ref contr}
\wu(t) = \left\{
\begin{matrix}
u_1, \quad & t \in [0, \wtau_1[ , \\
0, \quad & t \in ]\wtau_1, \wtau_2[,  \\
u_3, \quad & t \in ]\wtau_2, T], \\
\end{matrix} \right.
\end{equation} 
is called the \emph{reference pair}; $\wxi$ and $\wu$ are the \emph{reference trajectory} and the \emph{reference control}, respectively.

In analogy to the classical bang-bang case, we say that $\wtau_1$ and $
\wtau_2$ are the {\em switching times} of the reference control $\wu$ and we call $\wxi(\wtau_1)$, $\wxi(\wtau_2)$ the {\em switching points} of $\wxi$.
For the sake of future notation we set $\wtau_0 := 0$, $\wtau_3 := T$ and $\wI_j := (\wtau_{j-1}, \wtau_{j} )$, $j = 1, 2, 3$.
Moreover, we call $h_j$ the vector field that defines the reference trajectory in the interval $\wI_j$, that is, 
\begin{equation}
 h_1 := f_0+u_1f_1, \qquad h_2 := f_0, \qquad  h_3 :=f_0+u_3f_1,
\end{equation}
and the
\emph{reference time-dependent vector field} $\wh h_t=f_0+\wu(t)f_1$, so that  $\wh h_t\equiv h_j$ for $t\in \wI_j$. Denote by  $\widehat{S}_t(x)$  the solution at time $t$ of the
Cauchy problem 
\begin{equation}
\label{eq:refdynam}
\dot{\xi}(t)=\wh h_t(\xi(t)), \qquad 
\xi(0)=x.
\end{equation}

In this paper, we provide some first and second  order sufficient conditions 
that guarantee the \emph{strong-local} optimality of $\wxi(\cdot)$, accordingly to the following definition:
\begin{definition}
A curve $\wxi \colon [0,T]\to M$, solution of  \eqref{eq: contr sys}-\eqref{eq: vincoli} with associate control $\wu(\cdot)$, is a strong-local minimizer 
of \eqref{problema} if  
there exists a neighborhood $\mathcal{U}$ in $[0,T] \times M$ of the graph of $\wxi$ such 
that $ \int_0^T |\wu(t)\psi(\wxi(t))| dt\leq  \int_0^T \abs{u(t)\psi(\xi(t))} dt$ for every admissible trajectory $\xi \colon [0,T]\to M$ of  
\eqref{eq: contr sys}-\eqref{eq: vincoli}-\eqref{eq: controllo}, with associated control $u(\cdot)$ and whose graph is in $\cU$.

If  $\int_0^T |\wu(t)\psi(\wxi(t))| dt < \int_0^T \abs{u(t)\psi(\xi(t))} dt$ for any admissible trajectory other than $\wxi$ and whose graph is in $\cU$, we say that $\wxi$ is a {\em strict} strong-local 
minimizer.
\end{definition}


\subsection{Assumptions}

Our first  assumption ensures that the cost function is smooth along the reference trajectory but for a finite number of points. It also 
guarantees that the Hamiltonian vector field associated with the maximized Hamiltonian (defined in \eqref{eq:Hmax} below)
is well defined and $C^1$ but in a codimension $1$ subset,
 as explained in Remark~\ref{connected zero} below. 
It is also crucial in defining the second variation of an appropriate subproblem of \eqref{problema}.
\begin{hypo}\label{NT} 
For every $t\in[0,T]$, the vector field $\wh h_t$ is never tangent to the set $\{\psi=0\}$ along $\wxi(t)$, that is
 $L_{\wh h_t}\psi(\wxi(t))\neq 0$. 
 \end{hypo}
Thanks to 
this assumption,
$\psi\circ\wxi$ vanishes at most a finite number of times.  In particular, we assume that $\psi(\wxi(t))$ 
vanishes $n_1$ times in the interval $\wI_1$ and $n_3$ times in the interval $\wI_3$ (where we admit the cases $n_1 =0$ and $n_3 = 0$). 
We denote as
$0 < \sss_{11} < \sss_{12}< \ldots < \sss_{1 n_1 -1 } < \sss_{1 n_1} < \wtau_1 $,
the zeros of $\psi\circ\wxi$ occurring in $\wI_1$ and  
$\wtau_2 < \sss_{31} < \sss_{32}< \ldots < \sss_{3 n_3 -1 } < \sss_{3 n_3} <T $,
the zeros of $\psi\circ\wxi$ occurring in $\wI_3$.
Set $a_0=\lim_{t \to 0^+}\mathrm{sgn}(\psi(\wxi(t)))$ and $a_2=\mathrm{sgn}(\psi(\wxi(\wtau_2)))$. In particular,
\begin{alignat*}{2}
\mathrm{sgn}\big(\psi\circ\wxi|_{(\sss_{1\,i-1},\sss_{1i})}\big)&=a_0 (-1)^{i-1}, \quad && i\in\{1,\ldots,n_1+1\},\\
\mathrm{sgn}\big(\psi\circ\wxi|_{(\sss_{3\,i-1},\sss_{3i})}\big)&=a_2 (-1)^{i-1}, \quad && i\in\{1,\ldots,n_3+1\},
\end{alignat*}
(here we set $\sss_{10}=0$, $\sss_{1\,n_1+1}=\wtau_1$, $\sss_{30}=\wtau_2$, $\sss_{3\,n_3+1}=T$).
\begin{hypo} \label{SS}
The cost $\psi$ does not vanish at the switching points of the reference trajectory, that is $\psi(\wxi(\wtau_i))\neq 0,\ i=1,2$. 
\end{hypo} 
\begin{remark}\label{connected zero}
By continuity and Assumption~\ref{NT},  there exist two neighborhoods $\mathcal{V}_1,\mathcal{V}_3$ in $M$ of the arcs $\left. \wxi\right\vert_{\wI_1},\ \left. \wxi\right\vert_{\wI_3}$
respectively, such that the zero level sets of $\left.\psi \right\vert_{\mathcal{V}_1}$ and $\left.\psi \right\vert_{\mathcal{V}_3}$ are codimension-$1$ submanifolds  of $M$, transverse to $h_1$ and $h_3$, respectively.
Moreover the zero level set of $\left.\psi \right\vert_{\mathcal{V}_j}$, $j=1,3$,  has $n_j$ connected components.	
\end{remark}
Our next assumption is that the reference pair satisfies the Pontryagin Maximum Principle (PMP from now on) in its normal form; 
in principle, non-smooth versions of PMP
(see e.g. \cite[Theorem 22.26]{Clarke}) are required. Indeed our problem can be seen as a hybrid control problem, as defined in 
\cite[Section 22.5]{Clarke}, with the switching surface $S$ given by $\left\{(t, x, y) \colon \psi(x) =0, \ y = x \right\}$. 

Nevertheless, thanks to Assumption~\ref{NT}, in the case under study \cite[Theorem 22.26]{Clarke} reduces to the standard smooth version of PMP (as stated, for instance, in \cite{AgSac}).

In order to apply PMP,  
consider $\Psi\colon \ell \in T^*M\mapsto \psi\circ\pi(\ell)) \in \mathbb{R}$. As we shall extensively use it, we recall that  $\vPsi(\ell) = (-D\psi(\pi\ell), 0).$ Define the (normal) \emph{control-dependent Hamiltonian}, 
\begin{equation}
\mathfrak{h}(u,\ell) := F_0(\ell) +u F_1(\ell) - |u \Psi(\ell)|,   \qquad (u,\ell)\in[-1,1]\times T^*M \label{eq:ctr dep Ham} 
\end{equation}
and, for $\ell \in T^*M$, the \emph{reference Hamiltonian}  
 and the \emph{maximized Hamiltonian}  respectively as 
\begin{align}
\widehat{H}_t (\ell)&:=F_0 +\wu(t) F_1(\ell) - |\wu(t) \Psi(\ell)|,  \label{eq:ref Ham}\\
H^{\max}(\ell) \! &:= \!\!\!\max_{u\in[-1,1]}\!\!\mathfrak{h}(u,\ell) \!=\! 
\begin{cases}
F_0(\ell) + F_1(\ell) - \abs{\Psi(\ell)},  & \text{if }  F_1(\ell) > \abs{\Psi(\ell)} , \\
F_0(\ell ), & \text{if } \abs{F_1(\ell)} \leq \abs{\Psi(\ell)} , \\
F_0(\ell) - F_1(\ell) - \abs{\Psi(\ell)},  & \text{if } F_1(\ell) < -  \abs{\Psi(\ell)} , 
\end{cases} \label{eq:Hmax}
\end{align}
%
%
and we assume the following.
\begin{hypo}[PMP] \label{PMP} There exists a Lipschitzian curve $\wla\colon [0,T] \to T^*M$  such that
\begin{alignat}{2}
  \dot{\wla}(t)&=\vecc{\widehat{H}_t} (\wla(t)) \quad &\text{a.e. } t &\in [0,T],\nonumber\\
 \pi \wla(t)&= \wxi(t), &  \forall t &\in [0,T] . \nonumber \\
 \widehat{H}_t (\wla(t)) &= 
H^{\max}(\wla(t)), \qquad &  \qo t &\in [0,T] . \label{PMP max}
\end{alignat}
\end{hypo}
The curve $\wla(\cdot)$ is called the \emph{reference extremal}. In the following, we set 
\[
\well_0=\wla(0), \quad \well_1 = \wla(\wtau_1), \quad \well_2 = \wla(\wtau_2),
\quad \well_T = \wla(T) .
\]
Let $\xi(\cdot)$ be an admissible trajectory of \eqref{eq: contr sys} that satisfies the Pontryagin Maximum Principle  
with associated extremal $\lambda(\cdot)$.
If the costate
 $\lambda(t)$ does not belong to the set
\begin{equation} \label{eq:switchsurf}
\left\{
\ell \in T^*M \colon \abs{F_1(\ell)} = \abs{\Psi(\ell)}
\right\},
\end{equation}
then the control associated with $\xi$ can be recovered uniquely by equation \eqref{PMP max}.
Analogously to the smooth case, we call such a set the {\em switching surface}  of problem \eqref{problema}.
\begin{remark} 
Assumption \ref{SS} ensures that in a neighborhood of the range of $\wla(\cdot)$ the switching surface is composed by two non-intersecting connected components, 
$\left\{ F_1 = \abs{\Psi} \right\}$ and $\left\{ F_1 = - \abs{\Psi} \right\}$. 
Together with Assumption~\ref{NT}, it guarantees that the Hamiltonian vector field associated with $\widehat{H}_t$ is well defined along the reference trajectory,
except at the times $t=\sss_{1 i}$, $i=1, \ldots, n_1$, $\sss_{3i}$, $i=1, \ldots, n_3$, $\wh\tau_1, \wh\tau_2$, where it has possibly different left-sided and right-sided limits. 
Assumption \ref{NT} can actually be weakened: the non-tangency condition is required only along the two bang arcs. 
\end{remark}

It is easy to see that the reference Hamiltonian~\eqref{eq:ref Ham}, as a function on the cotangent bundle, is  piecewise time-independent, and it takes the values 
\begin{equation} \label{H1H2H3}
\widehat{H}_t (\ell):=
\begin{cases}
H_1^{\sigma_{0i}}:=F_0+u_1F_1+\sigma_{0i}\psi\circ\pi, & t\in[\sss_{1\,i-1},\sss_{1i}],   \\[-1mm]
 & \quad\qquad i=1,\ldots,n_1+1,\\
H_2:=F_0,  &  t\in[\wtau_1,\wtau_2],\\
H_3^{\sigma_{2i}}:=F_0+u_3F_1+\sigma_{2i}\psi\circ\pi,  & t\in[\sss_{3\,i-1},\sss_{3i}], \\[-1mm]
& \quad\qquad i=1,\ldots,n_3+1,
\end{cases}
\end{equation}
where we used the following symbols:
\[
\sigma_{0i}=a_0(-1)^i,\ i=1,\ldots,n_1,\qquad
\sigma_{2i}=a_2(-1)^i,\ i=1,\ldots,n_3. 
\]
Because of the maximality condition of PMP (equation \eqref{PMP max}), the following inequalities hold along the reference extremal:
\begin{gather} \label{max debole 1}
u_1 F_1(\wla(t))\geq \big|\psi(\wxi(t))\big|,\ t\in \wI_1, \quad\quad
\big|F_1(\wla(t))\big|\leq \big|\psi(\wxi(t))\big|,\ t\in \wI_2,  \\[2mm]
u_3 F_1(\wla(t))\geq \big|\psi(\wxi(t))\big|,\ t\in \wI_3. \label{max debole 2}
\end{gather}
%
In addition,  the following relations must hold:
\begin{equation}
\label{eq:debole}
\frac{d}{dt}\big(H_2-H_1^{\sigma_{0 n_1}}\big)(\wla(t))|_{t=\wh\tau_1}\geq 0, \qquad 
\frac{d}{dt}\big(H_3^{\sigma_{20}}-H_2\big)(\wla(t))|_{t=\wh\tau_2}\geq 0.
\end{equation}
Even though the derivative $\wla(t)$ is discontinuous at $t= \wtau_i$, $i=1, 2$, the two derivatives in \eqref{eq:debole} exist and are well defined. In particular, they can be expressed as follows:
\begin{align}
\frac{d}{dt}\left.\!\big(H_2-H_1^{\sigma_{0 n_1}}\big)(\wla(t))\right\vert_{t=\wh\tau_1}
\!\! \!\! &=\Big\{H_1^{\sigma_{0 n_1}},H_2\Big\}(\well_1)  \! = \! \dueforma{\vH1^{\sigma_{0n_1}}}{\vH2}(\well_1), \label{eq:RS1eq} \\
\frac{d}{dt}\left.\!\big(H_3^{\sigma_{20}}-H_2\big)(\wla(t))\right\vert_{t=\wh\tau_2}
\!\! \!\! & = \Big\{H_2,H_3^{\sigma_{2 0}}\Big\}(\well_2) \! = \! \dueforma{\vH2}{\vH3^{\sigma_{20}}}(\well_2) . \label{eq:RS2eq}
\end{align}
\begin{remark}Due to the nonlinearity introduced by the absolute value of the control in the integral cost, two consecutive bang arcs can exist 
only if $\psi$ vanishes at the switching point between such arcs (this is a straightforward application of the PMP and of previous remarks). 
In particular, we relate this fact to \cite[Theorem 3.1]{maurer-vossen}, that states the impossibility of the existence of bang-bang junctions: indeed,
the results of \cite{maurer-vossen} hold for the case
$\psi\equiv 1$ (plus some other control-independent terms in the cost).
\end{remark}
As anticipated in the Introduction, in order 
to apply the Hamiltonian methods, we need the  maximized Hamiltonian 
$H^{\max}$ to be well defined and sufficiently 
regular in a neighborhood of the range 
of the reference extremal. 
More precisely, we need to ensure that the analytical expression of the maximized Hamiltonian is the same on a sufficiently small
neighborhood of 
$\wla(t)$, for any $t \neq \wtau_1, \wtau_2$. 
This is guaranteed by the strengthened version of the necessary conditions \eqref{max debole 1}-\eqref{max debole 2}
and \eqref{eq:debole}.
\begin{hypo}\label{RA} 
Along each arc the reference control is the only one that maximizes over $u \in [-1, 1]$ the control-dependent Hamiltonian evaluated along 
$\wla$, $\mathfrak{h}(u,\wla(t))$. Equivalently, in equations~\eqref{max debole 1}-\eqref{max debole 2} the strict inequalities hold true
for any $t \neq \wtau_1, \wtau_2$.
\end{hypo}
The following assumption is the  strict form of the necessary conditions \eqref{eq:RS1eq}-\eqref{eq:RS2eq}. 
It basically states that the change of value in the control $\wh u$ can be detected 
by the first order approximation of $\left( H_2-H_1^{\sigma_{0n_1}} \right)\circ\wla$ at $\wtau_1$ and $\left( H_3^{\sigma_{20}}-H_2 \right) \circ\wla$ at $\wtau_2$, 
respectively. It ensures the transversality of the intersection between the reference extremal $\wla$ and the switching surface \eqref{eq:switchsurf}, 
thus avoiding bang-singular junctions.
\begin{hypo}\label{RS} 
\[\dfrac{d}{dt}\big(H_2-H_1^{\sigma_{0n_1}}\big)(\wla(t))|_{t=\wh\tau_1}>0 , \qquad 
  \dfrac{d}{dt}\big(H_3^{\sigma_{20}}-H_2\big)(\wla(t))|_{t=\wh\tau_2}>0.\]
\end{hypo}
\subsection{The Second Variation} 
This section is devoted to the construction of the second variation associated with problem \eqref{problema}, in the spirit of \cite{ASZ02,PoggSte04,ASZ98,PoggTo}.

In order to compute the second variation of the cost functional, one should take into account all the admissible variations of the reference trajectory 
$\wxi$.
Actually, as  it will be proved,  it suffices 
to consider only the trajectories corresponding to controls having the same bang-zero-bang structure of the reference one, that is,
to allow only variations of the switching times $\wtau_1$ and $\wtau_2$. 
We stress that Assumption~\ref{RA} is crucial to justify this reduction.
We thus obtain the following two dimensional sub-problem of \eqref{problema}:
 \begin{equation} \label{ext cost}
 \min_{0<\tau_1<\tau_2<T} \int_{I_1 \cup I_3} \left\vert \psi(\xi(t))\right\vert \;dt
 \end{equation}
 subject to 
 \begin{equation}\label{eq: sub ctr sys}\left\{
\begin{array}{l}
\dot{\xi} = h_i \circ \xi(t) \quad t \in I_i, \quad i=1, 2, 3, \\
\xi(0)= \wbx_0, \quad \xi(T) = \wbx_f  , \\
I_1=(0,\tau_1), \ I_2=(\tau_1,\tau_2),\   I_3=(\tau_2,T), 
\end{array} \right. 
\end{equation}
It is well known that, on a smooth manifold, the second derivatives
are highly coordinate-dependent, unless they are computed with respect to variations contained in the kernel of the differential of the function to be derived. To overcome 
this problem and obtain an intrinsic expression of the second variation of the cost \eqref{ext cost}, we introduce two smooth functions $\alpha,\beta \colon M \to \mathbb{R}$ satisfying
\begin{equation} \label{eq: alpha beta}
d\alpha(\wbx_0)=\well_0, \qquad d\beta(\wbx_f)=-\well_T,
\end{equation}
and remove the constraint on the initial point of the trajectories, thus obtaining the extended sub-problem:
\begin{equation} \label{eq: costext}
\min_{\tau_1,\tau_2}  \Big( \alpha(\xi(0)) + \beta(\xi(T)) + \int_{I_1 \cup I_3} \left\vert \psi(\xi(t))\right\vert \;dt \Big)
\end{equation}
\noindent
subject to
\begin{equation} \label{eq:contrsysext}
\left\{
\begin{array}{l}
\dot{\xi} = h_i \circ \xi(t) \quad t \in I_i, \quad i=1, 2, 3, \\
\xi(0) \in M, \quad \xi(T) = \wbx_f  \\
I_1=(0,\tau_1), \ I_2=(\tau_1,\tau_2),\   I_3=(\tau_2,T).
\end{array} \right. 
\end{equation}
It is easy to see, by PMP, that the differential of the cost \eqref{eq: costext} is 
zero at $\wbx_0$. 

The necessity of summing variations belonging to tangents spaces based at different points of the reference trajectory is a typical issue in geometric control.  
We get around this problem by pulling-back the solutions of 
\eqref{eq:contrsysext}
to the initial point $\wbx_0$, by means of the flow of the reference vector field; in this way, the variations of the trajectory will 
evolve on the tangent space $T_{\wbx_0}M$.

To simplify the expressions of the pull-backs, we first reparametrize time in such a way that each interval $I_j$ is mapped into the corresponding reference interval $\wI_j$;  
to do that, we
consider the {\em variations of the lengths of the intervals}
\[
\theta_i := \left(\tau_i - \tau_{i-1} 
\right) - \left(\wh\tau_i - \wh\tau_{i-1} 
\right), 
\]
we set $\bk:= \begin{pmatrix}
\theta_1,\theta_2,\theta_3
 \end{pmatrix}$,
and define the piecewise-affine reparametrization $\varphi_{\bk} \colon [0,T] \to [0, T]$ by means of the Cauchy problem
\[
\varphi_{\bk}(0) = 0, \quad
\dot\varphi_{\bk}(t) = \dfrac{\tau_i - \tau_{i-1}}{\wh\tau_i - \wh\tau_{i-1}} 
\quad \forall t \in [\wtau_{i-1}, \wtau_{i}), \quad i=1, 2, 3.
\]
Denoting the \emph{pull-back fields} as 
\begin{equation} \label{pullback fields}
g_i := \wSinv{t\, *}\, h_i \circ \wS{t} , \quad t \in [\wh\tau_{i-1}, \wh\tau_i], \quad i=1, 2, 3,
\end{equation}
the {\em pullback system} of \eqref{eq:contrsysext} has the form
\[
\begin{cases}
 \dot\zeta_t(x, \bk) =
\frac{\theta_i}{\wtau_i-\wtau_{i-1}}
 g_i(\zeta_t(x, \bk))
\quad t \in [\wh\tau_{i-1}, \wh\tau_i), \quad i=1, 2, 3, \\
\zeta_0(x,\bk)=x ,
\end{cases}
\]
where we notice that $\dot{\varphi}_{\bk}(t)=1+\frac{\theta_i}{\wtau_i-\wtau_{i-1}}
$ for every $t\in \wI_i$.
Finally, setting
\begin{equation} \label{wbeta e wpsi}
\wh\beta := \beta\circ \wS{T}, \quad 
\wpsi{t} := \psi\circ \wS{t}, 
\end{equation}
the cost can be written as
\begin{equation}
J (x, \bk) = \alpha(x) + \wh\beta(\zeta_T(x, \bk)) + 
 \int_{\wI_1 \cup \wI_3} \dot{\varphi}_{\bk}(t)  \big|\wpsi{t}(\zeta_t(x, 
 \bk))\big| \;dt. 
\end{equation}
Since we need to compute the first and second variations of $J$, it is necessary to get rid of the absolute value inside the integral, that is, to locate the zeros of the function $\wpsi{t}(\zeta_t(x,\bk))$ for $t\in \wI_1\cup \wI_3$. They turn out to be smooth functions of the initial state $x$ and of the switching time variations $\bk$, as the following Lemma states.
\begin{lemma}
Locally around $(\wbx_0,\boldsymbol{0})$,  there exist $n_1$ smooth functions $\sg_{1i}$ of $(x,\bk)$ such that $\sg_{1i}(\wbx_0,\boldsymbol{0}) = \sss_{1i}$ and
\begin{equation} \label{sg1 implicit}
\wpsi{\sg_{1i}(x,\bk)}\circ\exp\Big(\sg_{1i}(x,\bk)\frac{\theta_1}{\wtau_1}g_1\Big)(x)=0, \qquad i=1,\ldots,n_1.
\end{equation}
Analogously,  for any $ i=1,\ldots,n_3$, there exists a smooth function $\sg_{3i}$ of $(x,\bk)$ such that
$\sg_{3i}(\wbx_0,\boldsymbol{0}) = \sss_{3i}$ and, 
\begin{equation} \label{sg2 implicit}
\wpsi{\sg_{3i}(x,\bk)}\circ\exp\Big((\sg_{3i}(x,\bk)-\wtau_2)\frac{\theta_3}{T-\wtau_2}g_3\Big)
\circ \exp\big(\theta_2 g_2\big)
\circ \exp\big(\theta_1 g_1\big)
(x)=0.
\end{equation}
Moreover, the following relations hold at the point $(\wbx_0,\boldsymbol{0})$:
\begin{gather*}
\frac{\partial \sg_{1i}}{\partial x}=-\,\frac{L_{\delta x} \wpsi{\sss_{1i}}(\wbx_0) }{L_{g_1} \wpsi{\sss_{1i}}(\wbx_0)} 
, \qquad 
\frac{\partial \sg_{1i}}{\partial \theta_1}=-\,\frac{\sss_{1i}}{\wtau_1}
, \qquad  
\frac{\partial \sg_{1i}}{\partial \theta_2}= \frac{\partial \sg_{1i}}{\partial \theta_3}= 0 ,   \\
\frac{\partial \sg_{3i}}{\partial x}=-\,\frac{L_{\delta x} \wpsi{\sss_{3i}}(\wbx_0) }{L_{g_3} \wpsi{\sss_{3i}}(\wbx_0)}
, \quad 
\frac{\partial \sg_{3i}}{\partial \theta_1}=-\,
\frac{L_{g_1} \wpsi{\sss_{3i}}(\wbx_0) }{L_{g_3} \wpsi{\sss_{3i}}(\wbx_0)}
, \\ 
\frac{\partial \sg_{3i}}{\partial \theta_2}=-\,
\frac{L_{g_2} \wpsi{\sss_{3i}}(\wbx_0) }{L_{g_3} \wpsi{\sss_{3i}}(\wbx_0)}
, \quad 
\frac{\partial \sg_{3i}}{\partial \theta_3}=-\frac{\sss_{3i}-\wtau_2}{T-\wtau_2} .
\end{gather*}
\end{lemma}
The Lemma is proved by  applying the implicit function theorem to the functions defined in \eqref{sg1 implicit} and \eqref{sg2 implicit}; 
this can be done thanks to Assumption~\ref{NT}.
As usual, we set $\sg_{10}=0,\ \sg_{1\, n_1+1}=\wtau_1, \ \sg_{30}=\wtau_2,\ \sg_{3\, n_3+1}=T$.

\smallskip

Taking advantage of the functions $\sg_{1i}, \ \sg_{3i}$, we can write the cost $J(x,\bk)$ as 
\begin{align*}
 J & (x, \bk) = \alpha(x) + \wh\beta(\zeta_T(x, \bk)) \\
&+ 
a_0  \Big(1+\frac{\theta_1}{\wtau_1}\Big) \sum_{i=1}^{n_1+1}(-1)^{i-1}
\!\!\! \!\!\! \!\!\!  \int\limits_{\sg_{1\, i-1}(x,\bk)}^{\sg_{1 i}(x,\bk)}  \!\!\! \!\!\!       \wpsi{t}\circ\exp\Big(
\frac{t\theta_1}{\wtau_1}g_1 
 \Big)(x)
  \;dt\\
  &+ 
a_2  \Big(1+\frac{\theta_3}{T-\wtau_2}\Big) \times \\
& \times \sum_{i=1}^{n_3+1}(-1)^{i-1} 
 \!\!\! \!\!\! \!\!\! \int\limits_{\sg_{3\, i-1}(x,\bk)}^{\sg_{3 i}(x,\bk)}  \!\!\!  \!\!\! 
 \wpsi{t}\circ
 \exp\Big(\frac{(t-\wtau_2)\theta_3}{T-\wtau_2}g_3\Big)
\circ \exp\big(\theta_2 g_2\big)
\circ \exp\big(\theta_1 g_1\big) (x)   dt,
\end{align*}
and compute its second variation 
at the point $(\wbx_0,\boldsymbol{0})$ along all variations $(\delta x,\beps)\in T_{\wbx_0}M\times \mathbb{R}^3$ 
compatible with the constraints \eqref{eq:contrsysext}, %
that is, along the variations in the space
\begin{equation} \label{admissible var}
V=\{(\delta x,\beps) \colon \delta x+\sum_{i=1}^3\ep_ig_i=0,\ \sum_{i=1}^3\ep_i=0 \}.
\end{equation}
After some easy but tedious manipulations, we end up with the following expression:
\begin{equation}
\label{eq:secondvar}
\begin{split}
& J\se[\dx, \beps]^2 =  \dfrac{1}{2}D^2\wh\gamma(\wbx_0)[\dx]^2 + 
 \liederder{\dx}{\sum\limits_{i=1}^3\ep_i g_i}{\wh\beta} 
 +  \liederdo{\sum\limits_{i=1}^{3}\ep_i g_i }{\wh\beta} 
 \\
& + a_2 \sum_{i=1}^{n_3} (-1)^i \dfrac{\lieder{\dx}{\wh\psi_{\sss_{3i}}}}{\lieder{g_3}{\wpsi{\sss_{3i}}}}\left(
2\lieder{\dx + \ep_1 g_1 + \ep_2 g_2}{\wpsi{\sss_{3i}}} - \lieder{\dx}{\wpsi{\sss_{3i}}}
\right) \\ 
& +a_0\sum_{i=1}^{n_1} (-1)^i \dfrac{\left( \lieder{\dx}{\wpsi{\sss_{1i}}} \right)^2}{\lieder{g_1}{\wpsi{\sss_{1i}}}} + \int_{\wI_3} \liederder{\dx}{\ep_1 g_1 + \ep_2 g_2}{\left\vert \wpsi{t} \right\vert} \, dt 
\\ %
& +a_0 \dfrac{\ep_1}{2} (-1)^{n_1}\left(
\lieder{\dx}{\wh\psi_{\wh\tau_1}}
+ \lieder{\dx + \ep_1 g_1}{\wh\psi_{\wh\tau_1}}
\right) \\ 
& +a_2(-1)^{n_3} \ep_3 \lieder{\dx + \sum\limits_{i=1}^3\ep_i g_i}{\wh\psi_T} 
- \dfrac{\ep_3^2}{2}\lieder{g_3}{\wh\psi_T}  \\ 
& + a_2 \sum_{i=1}^{n_3} (-1)^i \dfrac{\left( \lieder{\ep_1 g_1 + \ep_2 g_2}{\wh\psi_{\sss_{3i}}} \right)^2}{\lieder{g_3}{\wh\psi_{\sss_{3i}}}}  + \dfrac{1}{2}\int_{\wI_3} \liederdo{\ep_1 g_1 + \ep_2 g_2}{\left\vert \wh\psi_{\wh t} \right\vert} \, dt \\ 
& +
 \dfrac{\ep_1 \ep_2}{2} \int_{\wI_3} \lieder{\liebr{g_1}{g_2}}{\left\vert \wh\psi_{t}\right\vert} \, dt 
  + \dfrac{1}{2}\sum_{1 \leq i< j \leq 3}\ep_i \ep_j \lieder{\liebr{g_i}{g_j}}{\wh\beta},
\end{split}
\end{equation}
where $\wh\gamma(x) := \alpha(x) + \wh\beta(x) + \int_{\wI_1 \cup \wI_3} |\wpsi{s}(x)| ds$. 
\begin{remark}
The terms containing the Lie derivative of $|\wpsi{t}|$ are defined everywhere, except for a finite number of values of $t$, therefore the integrals are well defined. 
\end{remark}
We recall that the space of admissible variations for the original subproblem \eqref{ext cost}-\eqref{eq: sub ctr sys} is
\[
V_0=\{(\delta x,\beps) \colon \delta x+\sum_{i=1}^3\ep_ig_i=0,\ \sum_{i=1}^3\ep_i=0, \ \delta x  = 0 \}.
\]

Indeed, 
only the restriction to $V_0$ of \eqref{eq:secondvar} is independent of the choice of $\alpha$ and $\beta$,
while its value on the whole $V$ depends on the choice of these two functions. 
Moreover we notice that the second variation $J\se$ can be written as a sum: $J\se=\dfrac{1}{2}D^2\wh\gamma(\wbx_0)[\dx]^2+J\se_0$, where $J\se_0$ does not depend on $\alpha$. 
Since $D^2\wh\gamma(\wbx_0)[\dx]^2|_{V_0}$ is null, 
then, if $J\se|_{V_0}$ is coercive, we can always choose $\alpha$ in such a way that 
$J\se$ is coercive on the whole $V$.
In Section \ref{sec: main} we will prove that the coerciveness of $J\se$ on $V$ implies the 
existence of a suitable Lagrangian submanifold
such that 
the projection of the maximized Hamiltonian flow is an invertible map between this manifold and its image on $M$  which is one of the 
key points for applying Hamiltonian methods. In view of this, our last assumption is the following:
 \begin{hypo}\label{hp:secvar}
 $J\se[0, \beps]^2$ is coercive on $V_0$. 
 \end{hypo}
 \begin{remark}
We stress  that,  in order to apply the Hamiltonian methods, we need to know the expression of the second variation, equation \eqref{eq:secondvar}, of the extended sub-problem.
Nevertheless, the invertibility is guaranteed under Assumption~\ref{hp:secvar}, i.e.~it suffices  to check the coerciveness of the quadratic form 
\begin{equation}
\label{eq:secondvar red}
\begin{split}
 J\se[0, \beps]^2 &=  
 a_0 \dfrac{\ep_1}{2} (-1)^{n_1} \lieder{ \ep_1 g_1}{\wh\psi_{\wh\tau_1}}
 - \dfrac{\ep_3^2}{2}\lieder{g_3}{\wh\psi_T}   \\%
& + a_2\,  \ep_3^2 \sum_{i=1}^{n_3} (-1)^i  \lieder{g_3}{\wh\psi_{\sss_{3i}}} 
+ \dfrac{1}{2}\int_{\wI_3} \liederdo{\ep_1 g_1 + \ep_2 g_2}{\left\vert \wh\psi_{\wh t} \right\vert} \, dt \\ 
& +  \dfrac{\ep_1 \ep_2}{2} \int_{\wI_3} \lieder{\liebr{g_1}{g_2}}{\left\vert \wh\psi_{t}\right\vert} \, dt   + \dfrac{1}{2}\sum_{1 \leq i< j \leq 3}\ep_i \ep_j \lieder{\liebr{g_i}{g_j}}{\wh\beta}.
\end{split}
\end{equation}
on the triples $\beps \in \R^3$ satisfying 
\[
\ep_1+\ep_2+\ep_3=0\qquad
\ep_1g_1(\wbx_0)+\ep_2g_2(\wbx_0)+\ep_3g_3(\wbx_0)=0.\]
Thus, if  the two vectors $(g_3-g_2)(\wbx_0)$ and
$(g_2-g_1)(\wbx_0)$ are linearly independent, then the space $V_0$ is trivial, so that the second variation is coercive by definition. Otherwise, $V_0$ is a one-dimensional linear space.
\end{remark}
\section{Hamiltonian Formulation} \label{sec:Ham}
\subsection{Construction of the Maximized Hamiltonian Flow} \label{sec: cstr max Ham}
The maximized Hamiltonian \eqref{eq:Hmax} is well defined, continuous and piecewise smooth in $T^*M$. 
By contrast, its associated Hamiltonian vector field is not well defined on the switching surfaces and on the zero level set of $\Psi$. The scope of this section is to prove that, nevertheless, the flow associated with the maximized Hamiltonian is well defined, at least in a tubular neighborhood of the graph of the reference extremal.
This property relies on the fact that the reference extremal crosses transversally the hypersurfaces of discontinuity of the Hamiltonian vector field.

Fix $\epsilon>0$;
then on a suitable neighborhood of the range of $\wla|_{[0,\wtau_1-\epsilon]}$ 
the maximized Hamiltonian vector field is ill-defined on the zero level set of $\Psi$, else it coincides with $\vF{0} + u_1 \vF{1} - \mathrm{sgn}({\psi})\vPsi$. 
By Assumption \ref{NT}, in this neighborhood the set $\{\Psi = 0 \}$ has $n_1$ connected components, transversal to $\wla(\cdot)$. 
By means of the implicit function theorem, we can express the hitting  times of the integral curves of $\vF{0} + u_1 \vF{1} - \mathrm{sgn}({\psi})\vPsi$
with each connected component as a smooth function of their starting point.
\begin{proposition}
\label{prop:s1} 
There exists a neighborhood $\mathcal{U}$ of $\well_0$ in $T^*M$ such that for any $ \ell \in \mathcal{U}$ there exists a unique $s_{11} = s_{11}(\ell)$ such that
$s_{11}(\well_0) = \sss_{11}$ and
\[
\Psi\circ \exp (s_{11}(\ell) \vH1^{\sigma_{01}})(\ell) = 0 \quad \forall \ell\in \mathcal{U}.
\]
\end{proposition}
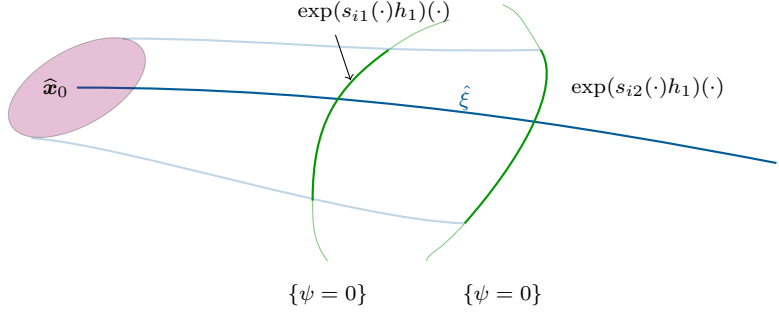
\begin{figure}
\begin{center}
  \begin{tikzpicture}
    \draw [color=blue8, thick](-.6, 0.5) .. controls (3, 0.5) and (6,0) .. (8.6, -.5); %

    \draw [color=green6, opacity=0.5](3.5, 1) .. controls (3.85, 1.25) and (4, 1.3) .. (4.3, 1.4); %

    \draw [color=green6, thick](3.5, 1) .. controls (2.8, 0.5) and (2.5,0) .. (2.5, -1); 

    \draw [color=green6, opacity=0.5](2.7, -1.8) .. controls (2.5, -1.55) and (2.5,-1.3) .. (2.5, -1); %

        \draw [color=green6, opacity=.5](5.5, 1) .. controls (5.2, 1.5) and (5.2,1.5) .. (5, 1.6); 

        \draw [color=green6, thick](5.5, 1) .. controls (5.8, 0.5) and (5.2,-.5) .. (4.5, -1.3); 

        \draw [color=green6, opacity =.5](4.5, -1.3) .. controls (4.15, -1.7) and (4,-1.7) .. (4, -1.8); 
   \draw [color=blue8, opacity=.25, thick] (-1.2, -.17) .. controls (0, -.25) and (3,-1.25) .. (4.5, -1.3); %
   
   \draw [color=blue8, opacity=.25, thick] (0, 1.15) .. controls (2, 1.15) and (3,.95) .. (5.5, 1); %


   \draw[fill=red8, opacity=.25] (-.6, .5) circle [x radius=1cm, y radius=5mm, rotate=30]; %
   
  \node[anchor=east] at (-.6,0.5){$\wbx_{0}$}; %
    \node[anchor=south] at (4.5,0.1){${\color{blue8}\hat\xi }$}; %
    \node[anchor=south] at (2.7,-2.5){$\{\psi=0\}$}; %
    \node[anchor=south] at (5,-2.5){$\{\psi=0\}$}; %

    \node[anchor=west] at (2.2,1.5){$\exp (s_{i1}(\cdot)h_1)(\cdot)$}; %
   \draw [->] (2.8, 1.2) -- (3, .6); %

    \node[anchor=west] at (5.8,0.5){$\exp (s_{i2}(\cdot)h_1)(\cdot)$}; %
  \end{tikzpicture}
\end{center}\caption{The times $s_{1i}$, $i=1, 2$}
\end{figure}
\begin{proof}
Since $\Psi\circ\exp (t \vH1^{\sigma_{01}})(\ell)  = \psi\circ\exp (t h_1 )(\pi\ell)$, it suffices to notice that 
\[
\left. \dfrac{\partial}{\partial t}  \psi\circ\exp (t h_1 )(\pi\ell)\right\vert_{(\sss_{11}, \well_0)} = \liede{h_1}{\psi}{\wxi(\sss_{11})}.
\] 
Thus, Assumption \ref{NT} yields the result. \hfill $\square$
\end{proof}
The curves   $\exp(t\vH1^{\sigma_{01}})(\ell)$ cross transversally the zero 
level set of $\Psi$ (thanks to the fact that they project on the integral curves of $f_0+u_1 f_1$ and to Assumption \ref{NT}), so that  
for $t$ in a sufficiently small right neighborhood of $s_{11}(\ell)$ the maximized Hamiltonian on the points $\exp(t\vH1^{\sigma_{01}})(\ell)$
is  $H_1^{\sigma_{02}}$, that is, $\vH1^{\sigma_{01}}$ is no more the vector field associated with the maximized Hamiltonian.
On the other hand, $H_1^{\sigma_{02}}$ is the maximized Hamiltonian also at the points
\begin{equation} \label{eq: flusso dopo s11}
\exp((t-s_{11}(\ell))\vH1^{\sigma_{02}})\circ\exp(s_{11}(\ell)\vH1^{\sigma_{01}})(\ell), \quad \ell\in \mathcal{U}.
\end{equation}
Then, for $t$ in a right neighborhood of $s_{11}(\ell)$, 
the maximized Hamiltonian flow is \eqref{eq: flusso dopo s11}.
Iterating the same argument as above, we can prove the following
\begin{proposition}
\label{prop:s1i} For every $i=1,\ldots,n_1$ and possibly shrinking the neighborhood $\mathcal{U}$, 
there exists a unique smooth function $s_{1i} \colon \mathcal{U} \to \mathbb{R}$ such that
\[
\Psi\circ 
\exp (s_{1i}(\ell)- s_{1\, i-1}(\ell)) \vH1^{\sigma_{0\, i}} \circ \cdots \circ 
\exp (s_{11}(\ell) \vH1^{\sigma_{01}})(\ell) = 0 \quad \forall \ell\in \mathcal{U}
\]
and $s_{1i}(\well_0) = \sss_{1i}$.

Moreover, the differential of $s_{1i}$ at $\well_0$ is given by
\[
\langle d s_{1i} (\well_0),\delta \ell \rangle =
-\,\frac{L_{\pi_*\delta \ell} \wh\psi_{\sss_{1i}}(\wbx_0)}{L_{g_1} \wh\psi_{\sss_{1i}}(\wbx_0)}.
\]
\end{proposition}
Taking advantage of the functions $\sss_{1i}$ defined here above, 
we can write the maximized Hamiltonian flow  on the set $\{(t,\ell)\in \mathbb{R}\times \mathcal{U} \colon 0\leq t\leq s_{1n_1}(\ell)\}$ 
as the following concatenation:
\begin{equation}\label{eq: flusso prima s11}
\mathcal{H}_t(\ell)= \exp\big((t-s_{1\, i-1}(\ell) )\vH1^{\sigma_{0i}}\big)\circ
\mathcal{H}_{s_{1\, i-1}(\ell)}(\ell) \qquad \begin{matrix} t\in[s_{1\, i-1}(\ell), s_{1 i}(\ell)], \\
i=1,\ldots,n_1,
\end{matrix}
\end{equation}
where we defined $s_{10}(\ell)\equiv \sss_{10}=0$. 
\begin{remark}
Actually, the functions $s_{1i}$ depend only on the projection $\pi \ell$. We write them as functions of $\ell$ for symmetry with other functions $s_{3i}$ that 
will be defined below. 
\end{remark}
Reasoning as above, we see that, for $t$ in a right neighborhood of $s_{1n_1}(\ell)$, the maximized Hamiltonian at the points
 $\exp (t- s_{1 n_1}(\ell)) \vH1^{\sigma_{0\,n_1+1}} \circ
\mathcal{H}_{s_{1n_1}(\ell)}(\ell)$
is $H_1^{\sigma_{0\,n_1+1}}$; in particular, $H_1^{\sigma_{0\,n_1+1}}$ is the maximized Hamiltonian along its integral curves until such curves 
intersect the hypersurface 
$\{H_2 - H_1^{\sigma_{0\,n_1+1}\!}=0\}$. 
As above,
thanks to the regularity assumptions (in this case Assumption \ref{RS}), we can  characterize the intersection time as a smooth function of the initial point $\ell$.
\begin{proposition}
\label{prop:tau1}
Possibly shrinking $\cU$, for any $\ell \in \cU$ there exists a unique $\tau_1 = \tau_1(\ell)$ such that
$\tau_1(\lo) = \wh\tau_1$ and
\[
\big(H_2-H_1^{\sigma_{0\, n_1+1}}\big)\circ\exp \big((\tau_1(\ell)- s_{1 n_1}(\ell)) \vH1^{\sigma_{0\,n_1+1}}\big) \circ
\mathcal{H}_{s_{1n_1}(\ell)}(\ell)=0
\]
Moreover, $\tau_1(\ell)$ is smooth and its differential at $\well_0$ is 
\begin{multline}
\label{eq:dtau1}
 \scal{d\tau_1(\lo)}{\dl} = 
 \dfrac{1}{\dueforma{\vH{1}^{\sigma_{0\, n_1+1}}}{\vH{2}}(\well_1)}\\
\times\Big\{
- \, \dueforma{\cHref_{\wh\tau_1*}\dl}{(\vH{2} - \vH{1}^{\sigma_{0\,n_1+1}})(\well_1)} \\
 + 2\sum_{i=1}^{n_1}\sigma_{0i}\dfrac{ \lieder{\pi_*\dl}{\wh\psi_{\wh s_{1i}}}\, \lieder{g_2 - g_1}{\wh\psi_{\wh s_{1i}}}}
{\lieder{g_1}{\wh\psi_{\wh s_{1i}}}}
\Big\}
\qquad 
\end{multline}
\end{proposition}
\begin{proof}
The proof is analogous to that of Proposition~\ref{prop:s1}. It 
suffices to notice that 
\begin{multline}
\left.\frac{\partial}{\partial t}\big(H_2-H_1^{\sigma_{0\,n_1+1}}\big)\circ\exp \big((t- s_{1 n_1}(\ell)) \vH1^{\sigma_{0\,n_1+1}}\big) \circ
\mathcal{H}_{s_{1n_1}(\ell)}(\ell)\right\vert_{(\wtau_1,\well_0)}\\ %
=\dueforma{\vH{1}^{\sigma_{0\,n_1+1}}}{\vH{2}}(\well_1),
\end{multline}
which is positive by Assumption~\ref{RS}. \hfill $\square$
\end{proof}

We extend the maximized Hamiltonian flow $\mathcal{H}_t$ to the whole first bang interval $[0,\tau_1(\ell)]$ setting 
\[
\mathcal{H}_t(\ell)=\exp \big((t- s_{1 n_1}(\ell)) \vH1^{\sigma_{0\,n_1+1}}\big) \circ
\mathcal{H}_{s_{1n_1}(\ell)}(\ell), \qquad t\in[s_{1n_1}(\ell),\tau_1(\ell)]. 
\]
The construction of the maximized Hamiltonian flow on the whole interval $[0,T]$ follows the same lines: we characterize 
the discontinuities of the vector field as smooth functions of the initial state, 
and then we concatenate the corresponding Hamiltonian flows (we are providing more details in Appendix~\ref{app tau}). The regularity assumptions are, as usual, crucial for this.
In particular, for initial conditions belonging to a suitable small neighborhood $\mathcal{U}$ of $\well_0$, the  maximized flow $\mathcal{H}_t(\ell)$
is defined by
\begin{equation} \label{max Ham}
\begin{cases}
\exp \big((t- \tau_1(\ell)) \vH2\big) \circ
\mathcal{H}_{\tau_1(\ell)}(\ell), & t\in[\tau_1(\ell),\tau_2(\ell)] \\
\exp \big((t- \tau_2(\ell)) \vH3^{\sigma_{20}}\big) \circ
\mathcal{H}_{\tau_2(\ell)}(\ell), & t\in[\tau_2(\ell),s_{31}(\ell)] \\
 \exp\big((t-s_{3\, i-1}(\ell) )\vH3^{\sigma_{2i}}\big)\circ
\mathcal{H}_{s_{3\, i-1}(\ell)}(\ell) \quad &  t\in[s_{3\, i-1}(\ell),s_{3 i}(\ell)] \\[-1mm]
& \quad\qquad i=1,\ldots,n_3+1,
\end{cases}
\end{equation}
where we put $s_{30}(\ell):=\tau_2(\ell)$ and $s_{3\, n_3+1}(\ell)\equiv T$.
\subsection{Hamiltonian Form of the Second Variation}
\label{sec:Hamiform}
In this section, we propose an alternative representation of the second variation,  
more compact and easier to deal with. To do that, we establish an isomorphism between $T^*_{\wbx_0}M
\times T_{\wbx_0}M$ and $T_{\well_0}(T^*M)$, and we map the Hamiltonians defined in \eqref{H1H2H3} to some Hamiltonian functions $G\se_i$ defined on $T^*_{\wbx_0}M
\times T_{\wbx_0}M$; we then express the second variation $J\se$ in terms of these Hamiltonians. The new expression 
of the second variation highlights its links with  the Hamiltonian vector fields and with the maximized Hamiltonian flow.

First, we define the following  anti-symplectic isomorphism\footnote{in particular,
if $\varsigma$ is the standard symplectic  form on the product $T^*_{\wbx_0}M
\times T_{\wbx_0}M$, then $\varsigma=-\boldsymbol{\sigma}\circ (\iota,\iota).$}
 between the space  $T^*_{\wbx_0}M
\times T_{\wbx_0}M$ and the tangent bundle $T_{\well_0}(T^*M)$: 
\[
\iota(\delta p,\delta x)=(-\delta p+A[\delta x,\cdot],\delta x),
\]
where $A$ is the symmetric bilinear form on $T_{\wbx_0}M$ defined by
\begin{align*}
& A[\delta x,\delta y]:= D^2 \Big(
-\wh \beta - \int_{\wI_1\cup \wI_3} \big|\wh \psi_s\big| ds\Big)
(\wbx_0)[\delta x,\delta y] \\
& \! -\!  2 a_0 \! \sum_{i=1}^{n_1}
\! (-1)^i \frac{
\lieder{\delta x}{ \wh \psi_{\sss_{1i}}} 
\lieder{\delta y}{ \wh \psi_{\sss_{1i}}}}
{
\lieder{g_1}{ \wh \psi_{\sss_{1i}}}}
 \! -\!  2 a_2 \! \sum_{i=1}^{n_3} \!  (-1)^i\frac{
\lieder{\delta x}{ \wh \psi_{\sss_{3i}}}
\lieder{\delta y}{ \wh \psi_{\sss_{3i}} }}
{\lieder{g_3}{ \wh \psi_{\sss_{3i}}}}.
\end{align*}
We then set
\begin{equation}\label{eq:Gsecondo}
\vGse{1} := \iota^{-1}  \vH{1}^{\sigma_{01}}(\well_0), \;\; \vGse{2}:= \iota^{-1}  \widehat{\mathcal{H}}_{\wtau_{1}*}^{-1} \vH{2}(\well_{1}) , 
\;\; \vGse{3} := \iota^{-1}  \widehat{\mathcal{H}}_{\wtau_{2}*}^{-1} \vH{3}^{\sigma_{21}}(\well_{2}).    
\end{equation}
By computation (see Appendix~\ref{app H}), one can see that each $\vGse{i}$ is the constant  Hamiltonian vector field associated with the following linear Hamiltonian functions on $T^*_{\wxo}M \times T_{\wxo}M $:
\begin{align}
 \begin{split}
G_1\se  (\dep, & \dx) = \scal{\dep}{g_1(\wxo)} + \liederder{\dx}{g_1}{\left(\wh\beta + \int_{\wI_3}\big|\wh\psi_s\big|ds\right)} \\ 
 & + a_0 (-1)^{n_1} \lieder{\dx}{\wh\psi_{\wh\tau_1} }
+ 2 a_2 \sum_{i=1}^{n_3} (-1)^i \dfrac{ 
\lieder{g_1}{\wh\psi_{\wh s_{3i}}} \lieder{\dx}{\wh\psi_{\wh s_{3i}}}}
{\lieder{g_3}{\wh\psi_{\wh s_{3i}}}},  
\label{G1}
\end{split}\\
\begin{split}
G_2\se (\dep, & \dx) = \scal{\dep}{g_2(\wxo)} + \liederder{\dx}{g_2}{\left(\wh\beta + \int_{\wI_3}\big|\wh\psi_s\big|ds\right)}  \\
 &+2 a_2\sum_{i=1}^{n_3} (-1)^i \dfrac{
\lieder{g_2}{\wh\psi_{\wh s_{3i}}} \lieder{\dx}{\wh\psi_{\wh s_{3i}}}}
{\lieder{g_3}{\wh\psi_{\wh s_{3i}}}} , \label{G2}
\end{split}\\
\begin{split}
G_3\se (\dep, & \dx) = \scal{\dep}{g_3(\wbx_0)} + \liederder{\dx}{g_3}{\wh\beta}  + 
a_2 (-1)^{n_3}\lieder{\dx}{\wh\psi_{T}}.\label{G3}
\end{split}
\end{align}
Finally, define a one-form $\omega_0(\delta x,\cdot)  \in T^*_{\wbx_0}M $ by 
$\iota^{-1}d\alpha_*\dx = \left(\omega_0(\delta x,\cdot), \dx \right)$. 
Then 
the second variation $J\se$ can be written as
\begin{align} 
J\se[\de]^2 &= \dfrac{1}{2}\Big(
\ep_1 G\se_1(\omega_0(\dx), \dx)  
+\ep_2 G\se_2\big( (\omega_0(\dx), \dx) + \ep_1 \vGse{1} \big) \label{eq:JGG}\\
&+\ep_3 G\se_3\big( (\omega_0(\dx), \dx) + \ep_1 \vGse{1} + \ep_2 \vGse{2}
\big)\Big) \nonumber
\end{align}
for every $\delta e=(\delta x,\boldsymbol{\varepsilon})\in V$.
The proof is just a straightforward application of the definitions (see Appendix~\ref{app H} for more details).

\medskip
Taking advantage of this formulation of the second variation, we are able to write
two conditions that are equivalent to the coerciveness of $J\se$ on $V$. 
We recall that, given any linear subspace $W\subset V$, then $J\se$ is coercive on $V$ if and only if it is coercive both on $W$ and on the 
orthogonal complement to $W$ with respect to the bilinear symmetric form $\mathfrak{J}$ associated with $J\se$. 
Let for instance 
choose $W$ as the linear space
\begin{equation} \label{eq: W}
W=\{\delta e=(\delta x,\boldsymbol{\varepsilon}) \in V: \varepsilon_3=0 \}.
\end{equation}
(this choice will be useful for future computations).
It is easy to prove that 
\begin{gather}
W^{\perp_{\mathfrak{J}}}=\{(\delta x,\boldsymbol{\ep}) \in V \colon \ep_1 = \langle d\tau_1(\well_0) ,d\alpha_* \delta x\rangle \},  \label{eq:W orth}
\\
J\se[\delta e]^2 =\frac{\ep_3}{2} \big(G\se_3-G\se_2 \big) \big((\omega_0(\delta x),x) +
\ep_1 \vGse{1}  +\ep_2 \vGse{2}\big) 
\qquad 
 \forall \delta e\!=\!(\delta x,\boldsymbol{\ep})\!\in \! W^{\perp_{\mathfrak{J}}},
\end{gather}
see Appendix~\ref{app H} for details.
\section{Main Result} \label{sec: main}
In this section, we state and prove the main result of the paper.
\begin{theorem} \label{main theorem}
Let $\wxi\colon[0,T]\to M$ be an admissible trajectory for the control system \eqref{eq: contr sys}-\eqref{eq: vincoli}-\eqref{eq: controllo} that satisfies Assumptions~\ref{NT}--\ref{hp:secvar}.
Then, the trajectory $\wxi$ is a strict strong-local minimizer for the OCP~\eqref{problema}.
\end{theorem}
Theorem~\ref{main theorem} relies on the following result.
\begin{theorem} \label{cond suff}
Assume that the assumptions of Theorem~\ref{main theorem} are satisfied. Then there exist a neighborhood $U$ of $\wbx_0$ such that the set
\begin{equation}
\Lambda_0:=\{d\alpha(x)\colon x\in U\}
\label{eq:Lambda0}
\end{equation}
is a smooth simply-connected Lagrangian submanifold that contains $\well_0$ and, for every $t\in[0,T]$, the map
\begin{equation} \label{good proj}
\pi\mathcal{H}_t|_{\Lambda_0}
\end{equation}
is invertible onto a neighborhood of $\wxi(t)$ with piecewise-$C^1$ inverse.
\end{theorem}
\begin{proof}(of Theorem~\ref{main theorem})
First of all, we define the following subsets of $\mathbb{R} \times T^*M$
\begin{align*}
\mathcal{O}_{1i} &=\{(t, \ell)\colon \ell \in \mathcal{U}, \ s_{1\, i-1}(\ell)\leq t\leq  s_{1 i}(\ell)\} \quad i=1,\ldots,n_1+1,\\
\mathcal{O}_2 &=\{(t, \ell)\colon \ell \in \mathcal{U}, \ \tau_1(\ell)\leq t\leq  \tau_2(\ell)\},\\
\mathcal{O}_{3i} &=\{(t, \ell)\colon \ell \in \mathcal{U}, \ s_{3\, i-1}(\ell)\leq t\leq  s_{3 i}( \ell)\} \quad i=1,\ldots,n_3+1,
\end{align*}
and the flow $\mathcal{H}\colon [0,T]\times  \mathcal{U} \to \mathbb{R} \times T^*M$ 
\[
\mathcal{H}(t,\ell)=(t,\mathcal{H}_t(\ell)).
\]
We also define the sets $\Omega_{ij}=\mathcal{O}_{ij}\cap (\mathbb{R} \times \Lambda_0)$, $i=1,3$, and 
$\Omega_2=\mathcal{O}_2\cap (\mathbb{R} \times \Lambda_0)$, and we call $\Omega$ the union of all these sets. 
Notice that  the restriction of $\mathcal{H}$ to each of the $\mathcal{O}_{ij},\ \mathcal{O}_{2}$ (as well as to each of the $\Omega_{ij},\
\Omega_2$) is smooth.
Moreover, thanks to Theorem \ref{cond suff},  the map $\pi\mathcal{H}\colon(t,\ell) \mapsto (t,\pi\mathcal{H}(t,\ell))$ is invertible with piecewise-$C^1$ inverse.
The points of non differentiability occur when $(t,\ell)$ belongs to the intersections  $\Omega_{1 \, n_1+1}\cap \Omega_2$ and
$\Omega_2\cap \Omega_{31}$. Indeed, we notice that
\[
\pi\mathcal{H}_t(\ell)=
\begin{cases}
\exp(t h_1)(\ell) & t \in [0, \tau_1(\ell)],\\
\exp((t-\tau_1(\ell)) h_2)\circ \exp(\tau_1(\ell) h_1) (\ell) & t \in [\tau_1(\ell),\tau_2(\ell)],\\
\exp((t-\tau_2(\ell)) h_3)\circ \exp((\tau_2(\ell)-\tau_1(\ell)) h_2)\\
\qquad \qquad \qquad \qquad \qquad \qquad \circ \exp(\tau_1(\ell) h_1) (\ell) 
\quad & t \in [\tau_2(\ell),T],
\end{cases}
\]
so that, at  the first switching 
time $\wtau_1$, the 
piecewise linearization  $\pi_*\mathcal{H}_{\wh \tau_1*}$ is given by 
\begin{equation}
\begin{cases} \label{eq: linear tau1}
\widehat{S}_{\wh\tau_1*}(\pi_*\delta \ell) \qquad & \mbox{for } \langle d\tau_1(\well_0),\delta \ell\rangle \geq 0,\\
\widehat{S}_{\wh\tau_1*}(\pi_*\delta \ell+ \langle d\tau_1(\well_0),\delta \ell\rangle(g_1-g_2)(\wbx_0)  ) \qquad & \mbox{for } 
\langle d\tau_1(\well_0),\delta \ell\rangle \leq 0,
\end{cases} 
\end{equation}
while at the second switching time
 $\wtau_2$, the 
piecewise linearization  $\pi_*\mathcal{H}_{\wh \tau_2*}$ is given by %
\begin{equation}
\begin{cases} \label{eq: linear tau2}
\widehat{S}_{\wh\tau_2*}(\pi_*\delta \ell
+\langle d\tau_1(\well_0),\delta \ell\rangle(g_1-g_2)(\wbx_0)) \qquad & \mbox{for } \langle d\tau_2(\well_0),\delta \ell\rangle \geq 0,\\
\begin{matrix}
\widehat{S}_{\wh\tau_2*}(\pi_*\delta \ell+\langle d\tau_1(\well_0),\delta \ell\rangle(g_1-g_2)(\wbx_0)  \\
+\langle d\tau_2(\well_0),\delta \ell\rangle(g_2-g_3)(\wbx_0)) 
\end{matrix} \qquad & \mbox{for } \langle d\tau_2(\well_0),\delta \ell\rangle \leq 0.
\end{cases} 
\end{equation}
Let $\omega_{(t,\ell)}=s_{\ell}-\max_{\abs{u} \leq 1 }\mathfrak{h}(u,\ell)dt$ be
the Poincar\'e-Cartan one-form (in the following we omit the dependence on the basepoint $(t,\ell)$). Lemma~3.3 in \cite{SZ16} guarantees that the one-form $\mathcal{H}^*\omega$
is exact on each of the $\Omega_{ij},\ \Omega_2$, and therefore on the whole $\Omega$.

Let now $\xi \colon [0, T] \to M$ be any admissible trajectory  of the control system \eqref{eq: contr sys}-\eqref{eq: vincoli} 
whose graph is contained in $\pi \mathcal{H}(\Omega)$, and
let $v(t)$ be the associated control function.
Define moreover the curves $\lambda_0(t)$, $\lambda(t)$  by the equalities 
$(t, \lambda_0 (t)) := (\pi\cH)^{-1}(t, \xi(t))$ and
$\lambda(t) := \mathcal{H}(t, \lambda_0(t))$; in particular, $\pi\lambda(t)=\xi(t)$. Consider the two paths in $\Omega$
\begin{gather*}
\gamma=\{(t,\lambda_0(t)) \colon t\in[0,T]\} \qquad
\widehat{\gamma}=\{(t,\well_0)  \colon t\in[0,T]\}.
\end{gather*}
The concatenation of $\widehat{\gamma}(\cdot)$ with $\gamma(T - \cdot)$ is a closed path in $[0,T]\times \Lambda_0$, so that 
$\int_{\widehat{\gamma}}\mathcal{H}^*\omega=\int_{\gamma}\mathcal{H}^*\omega$. 
%
In particular,
$ \int_{\widehat{\gamma}}\mathcal{H}^*\omega=\int_0^T |\wu(t) \psi(\wxi(t))|\;dt$, while 
\begin{equation} \label{eq: integral}
\int_{\gamma}\mathcal{H}^*\omega=\int_0^T s_{\lambda(t)} - \max_{\abs{u} \leq 1 }\mathfrak{h}(u,\lambda(t))\,dt \leq  \int_0^T   |v(t) \psi(\xi(t))|\;dt,
\end{equation}
and this implies that $\wxi$ is a strong-local minimizer.

\medskip
Let us now assume that $\int_0^T |\wu(t) \psi(\wxi(t))|\;dt=\int_0^T |v(t) \psi(\xi(t))|\;dt$, that is, the equality holds
in equation \eqref{eq: integral}. This
 implies that for a.e.~$t\in[0,T]$
\begin{equation} \label{eq: =qo}
 v(t) F_1(\lambda(t)) - |v(t) \psi(\xi(t))|=\max_{w\in[-1,1]} \big(w F_1(\lambda(t)) - |w \psi(\xi(t))| \big) .
\end{equation}
By continuity, for $t$ small enough, $(t, \lambda_0(t))$ belongs to $\Omega_{11}$, hence the quantity $ w F_1(\lambda(t)) - |w \psi(\xi(t))| $ 
attains its maximum only for $w=u_1$; equation \eqref{eq: =qo} yields that $v(t)=u_1$ a.e., so that $\xi(t)=\wxi(t)$ as long as $(t, \lambda_0(t)) \in \Omega_{11}$,
that is, for $t\in[0,\sss_{11}[$. For $t$ in a sufficiently small right neighborhood of $\sss_{11}$, $(t,\xi(t))$ belongs to $ \Omega_{11}$ or  $ \Omega_{12}$; in
both cases, \eqref{eq: =qo} implies that $v(t)=u_1$. We can proceed iteratively and obtain that $v(t)=u_1$ a.e. and $\xi(t)=\wxi(t)$ for $t\in[0,\wtau_{1}[$.

For $t$ in a sufficiently small right neighborhood of $\wtau_1$, three cases are possible: $(t,\lambda_0(t))$  may belong to $\Omega_{2}\setminus \Omega_{1\,n_1+1}$, to  $ \Omega_{1\,n_1+1}\setminus \Omega_{2}$, or to the intersection
$\Omega_{1\, n_1+1}\cap \Omega_{2}$. 

In the first case, the maximized Hamiltonian is attained for $w=0$ only, so that, reasoning as above, we obtain that $v(t)=0$ and then $\xi(t)=
\wxi(t)$ for $t\leq \wtau_2$.
If $(t,\lambda_0(t)) \in \Omega_{1\,n_1+1}\setminus \Omega_{2}$, the maximized Hamiltonian is attained for $w=u_1$ only and then \eqref{eq: =qo} yields that $v(t)=u_1$. 
This is impossible, since, by
Assumption~\ref{RS} and by continuity, in a neighborhood of $\well_1$ the function $F_1(\ell)-|\psi(\pi\ell)|$ is strictly decreasing along 
the integral lines of  $\vH{1}^{\sigma_{0\, n_1+1}}$. 

In the last case, for $t$ in a  sufficiently small right  neighborhood of $\wtau_1$ it holds $t = \tau_1(\lambda_0(t))$ which implies that 
\begin{equation}
\label{eq:dtau11}
1 = \scal{d\tau_1(\lambda_0(t))}{\dot\lambda_0(t)} \text{ for }\qo t.
\end{equation}
Moreover, \eqref{eq: =qo} implies that $v(t)$ has the same sign of $u_1$, so that $f_0 + v(t)f_1$ is a convex combination of $h_1$ and $h_2$, and there exists some $\mu(t) \in[0,1]$ such that $\dot{\xi}(t) = \mu(t) h_1(\xi(t))+(1-\mu(t))h_2(\xi(t))$. 

By computations $\dot\xi(t) = h_1(\xi(t)) + (\pi\cH_t)_*\dot\lambda_0(t)$ for $\qo t$, so that
\begin{equation}
\label{eq:plodot}
\pi_*\dot\lambda_0(t) = (1 - \mu(t)) \pi_* (\pi\cH_t)^{-1}_* (h_2 - h_1)(\xi(t)) \ \qo t.
\end{equation}
By compactness, there exists a sequence $t_n \to \wtau_1^+$ such that $\mu(t_n) \to \overline \mu \in [0,1]$.
Passing to the limit in \eqref{eq:dtau11}-\eqref{eq:plodot}   we obtain that
\[
\pi_*\dot\lambda_0(t_n) \to (1 - \overline\mu)(g_2 - g_1)(\wbx_0)
\]
and 
$1 = (1 - \overline\mu)\scal{d\tau_1(\well_0)}{d\alpha_*(g_2 - g_1)(\wbx_0)}$.
In particular, $\overline\mu \in [0,1)$ and the variation 
$\de = \left( (g_2 - g_1)(\wbx_0), 1, -1, 0\right)$ belongs to $W$.
Thus, by \eqref{eq:JGG} and \eqref{G2-G1} we obtain that
\[
J\se[\de]^2 = \dfrac{1}{2} G_2\se(\vGse{1}) \frac{\overline\mu}{1 - {\overline\mu}}, 
\]
which cannot be positive, due to \eqref{eq:RS1eq} and \eqref{G2G1}.
This contradicts Assumption \ref{hp:secvar}; thus this case is not possible.
Therefore, we must conclude that $v(t)=\wu(t)$ for a.e. $t\in[0,\wtau_2]$. 

Analogous computations show that $v(t)$ coincides with the reference control almost everywhere in the interval $[0,T]$, that is, $\xi=\wxi$.  \hfill $\square$
 \end{proof} 

 \begin{proof}(of Theorem~\ref{cond suff})
By construction, the manifold $\Lambda_0$ defined in \eqref{eq:Lambda0} 
is a horizontal Lagrangian submanifold of $T^*M$ containing $\well_0$. 

The map $\pi\mathcal{H}_t$ is the concatenation of smooth invertible mappings (flows).  To check its invertibility at $\well_0$, it is sufficient to
consider it at the switching times only: indeed, the map is invertible at $\well_0$
for every $t< \wtau_1$, since it is the flow of the field $h_1$, while at $t=\wtau_1$ folding phenomena could appear; if they don't, then we can conclude that
$\pi\mathcal{H}_t$ is invertible on the whole $[0,\wtau_2[$. Using the same argument, we can see that 
if the map is invertible for every $t\leq \wtau_2$, then it is invertible for every $t\in[0,T]$.

To verify the invertibility at the switching times, we use 
Clarke's inverse function theorem (see \cite{clarke-inv}), that is, we prove that all convex combination of the ``left'' and ``right'' linearizations \eqref{eq: linear tau1}--\eqref{eq: linear tau2}   have full rank. 

More precisely, for the first switching time, we show by contradiction that there is no 
$a\in[0,1]$ and no $\delta\ell\in T_{\well_0}\Lambda_0$, $\delta\ell\neq 0$, such that
\begin{equation} \label{non inv tau1}
(1-a)\widehat{S}_{\wh\tau_1*}(\pi_*\delta \ell)+a
\widehat{S}_{\wh\tau_1*}(\pi_*\delta \ell+(g_1-g_2)(\wbx_0) \langle d\tau_1(\well_0),\delta \ell\rangle) =0 .
\end{equation}
Indeed, assume that \eqref{non inv tau1} holds true for some $a$ and some 
$\delta \ell\neq 0$. Then
$\langle d\tau_1(\well_0),\delta \ell\rangle \neq 0$, and, since
$\widehat{S}_{\wh\tau_1*}$ is an isomorphism, then  
\[
\pi_*\delta \ell + a \scal{d\tau_1(\well_0)}{\delta \ell}
(g_1-g_2)(\wbx_0) = 0
\]
i.e.~$\delta e = \big(\pi_*\delta \ell,  a \scal{d\tau_1(\well_0)}{\delta \ell}, -  a \scal{d\tau_1(\well_0)}{\delta \ell},0
\big) \in W$. 
As before, thanks to \eqref{G2G1}-\eqref{G2-G1} and Assumption~\ref{RS}, it is possible to prove that 
\begin{align*}
J\se[\delta e]^2 &= -\frac{a}{2} \scal{d\tau_1(\well_0)}{\delta \ell} (G\se_2-G\se_1) \big( (\omega_0(\delta x),\delta x) + \ep_1 \vGse{1}\big)\\
&= \frac{a}{2}(1-a)\scal{d\tau_1(\well_0)}{\delta \ell}^2 G\se_2(\vGse{1})\leq 0,
\end{align*}
which contradicts the coerciveness of the second variation on $W$. 

Analogously, the linearization of the maximized flow at time $\wh\tau_2$ is invertible if for every $\delta\ell\in T_{\well_0}\Lambda_0$  and 
for every $a\in[0,1]$ satisfying
\begin{equation} \label{inv tau2}
\pi_*\delta \ell+\langle d\tau_1(\well_0),\delta \ell\rangle (g_1-g_2)(\wbx_0)+a\langle d\tau_2(\well_0),\delta \ell\rangle (g_2-g_3)(\wbx_0) = 0
\end{equation}
it must be $\dl = 0$.
Indeed, if \eqref{inv tau2} holds, then the variation
\[
\delta e = \big(\delta x,\langle d\tau_1(\well_0),d\alpha_*\delta x\rangle,\langle d(a\tau_2-\tau_1)
(\well_0),d\alpha_*\delta x\rangle,-a\langle d\tau_2(\well_0),d\alpha_*\delta x\rangle\big)
\]
is admissible and it is contained in $W^{\perp_{\mathfrak{J}}}$. 
Again, using \eqref{G3G2}, we observe that 
\[
  J\se[\delta e]^2 = \frac{a}{2}(1-a)\scal{d\tau_2(\well_0)}{d\alpha_*\dx}^2,
  G\se_3(\vGse{2})\,
\]
which cannot be positive due to Assumption~\ref{RS}. This contradicts the coerciveness of the second 
variation on $V$. The Theorem is proved. \hfill $\square$
\end{proof}
\section{An Example}\label{sec:example}
In this section, we apply our result to the following OCP 
\[
\min_{|u(\cdot)|\leq 1} \int_0^T |u(t)x_2(t)|\;dt
\]
subject to the control system
\begin{equation}
\begin{cases} \label{sys boizot}
\dot{x}_1=x_2\\
\dot{x}_2=u-\alpha x_2 \qquad \qquad  \alpha>0\\
x_1(0)=0,\ x_2(0)=0\\
x_1(T)= X>0, \ x_2(T)=0.
\end{cases}
\end{equation}
It models the problem of minimizing the consumption  of an academic electric vehicle moving with friction along a flat road; it has been studied in details in \cite{boizot}.
 The minimum time needed to reach the target point $(X,0)$ from $(0,0)$ is 
$T_{min}=\frac{1}{\alpha}\log \Big(\frac{1+\sqrt{1-e^{-\alpha^2 X}}}{1-\sqrt{1-e^{-\alpha^2 X}}}\Big)$; so, if $T<T_{min}$, there is no admissible trajectory
of \eqref{sys boizot}, and for $T=T_{min}$ the only admissible one is the minimum time trajectory, which is bang-bang; therefore, in the following we consider the case $T>T_{min}$. 
In \cite{boizot} it is proved that, for every fixed $\alpha>0$, the structure of the optimal control depends on the final time $T$ and the final point $X$; in particular,
given
\[
T_{lim} =\frac{1}{\alpha} \log \Big((1+\sqrt{2}) e^{\alpha^2 X} -1+\sqrt{(1+\sqrt{2}) e^{\alpha^2 X} -1)^2-1}\Big),
\]
the following facts hold:
\begin{itemize}
 \item if $T_{min}<T\leq T_{lim}$ the optimal control has the form \eqref{ref contr}, with $u_1=1$, $u_3=-1$, and switching times
 \begin{align}
 \wtau_1&=\frac{1}{\alpha} \log \Big(\frac{1}{2}\big(1+e^{\alpha T}-\sqrt{1+2 e^{\alpha T} +e^{2\alpha T}-4e^{\alpha^2X+\alpha T}}\big) \Big) \label{ex wtau1}\\
 \wtau_2&=
 \frac{1}{\alpha}\log\big(1-e^{\alpha \wtau_1}+e^{\alpha T}\big).
 \label{ex wtau2}
 \end{align}
 \item if $T>T_{lim}$, then the optimal control has a bang-singular-zero-bang structure.
\end{itemize}

We restrict our analysis to the case $T_{min}<T\leq T_{lim}$.
By integrating the system \eqref{sys boizot}, and taking into account the expression of the switching times, 
it is easy to see that 
the cost function $\psi(x_1,x_2)={x_2}$ is non-negative on the whole interval $[0,T]$, and vanishes only at the endpoints; thus, in this case we have $a_0=a_2=1$ and  $n_1=n_3=0$,
and, in particular, Assumption~\ref{SS} is satisfied.

Assumption~\ref{NT} can be easily checked by computations, while the validity of Assumption~\ref{PMP} (the candidate trajectories satisfy the PMP in the normal form) 
is proved in \cite{boizot}. 
In particular, the adjoint variable $p_1$ is constant in time, and
\begin{gather*}
p_1=\frac{e^{-\alpha\wtau_2}(e^{\alpha\wtau_1}-1)(1+e^{2\alpha(\wtau_2-\wtau_1)})}{e^{\alpha(\wtau_2-\wtau_1)}-1}\\
p_2(0)= \frac{1}{\alpha}\big(1-e^{-\alpha \wtau_1}\big)\big(p_1-1+e^{-\alpha \wtau_1}\big).
\end{gather*}

\noindent
Let us now verify that the other assumptions are met as well.
 
 \medskip
 \noindent
 {\it Assumption~\ref{RA}.} Since $\psi$ is always non-negative along the candidate trajectory, we can get rid of the absolute values in \eqref{max debole 1}-\eqref{max debole 2}. 
 In particular,
 we have to check the following
 \[
 \begin{cases}
  p_2(t)-x_2(t)>0 & t\in]0,\wtau_1[\\
  p_2(t)+x_2(t)<0 & t\in]\wtau_2,T[.
  \end{cases}
 \]
The weak inequalities hold by PMP. In particular,
$p_2(0)>x_2(0)=0$.
{
Assume that there exists some $\overline t \in [0, \wtau_1[$ such that $p_2(\overline t)=x_2(\overline t)$; then $\overline t $ is a minimum time for the function $\Delta  \colon t \in [0, \wtau_1 ] \mapsto p_2(t) - x_2 (t) \in \R$. We thus have
\begin{equation}
\Delta (\overline{t}) = 0, \qquad \dot \Delta (\overline{t}) = 0, \qquad \ddot \Delta (\overline{t}) \geq 0.
\end{equation}
Since $\ddot \Delta (t) = \alpha \left( 2 - p_1 + \alpha\Delta(t) \right) \geq \ddot \Delta (\overline{t})$ for any $t \in [0, \wtau_1]$, and $\Delta (\wtau_1) = 0$, we must have $\Delta (t ) \equiv 0$. This implies $p_1 = 2$ and $x_2(t) \equiv \frac{1}{\alpha}$, a contradiction.
%
This proves that Assumption 4 holds in the first bang arc.

The same reasoning applies in the second bang arc. By PMP, for every $t \in [\wtau_2,T]$ we have  $p_2(t)+x_2(t)\leq 0$ . Assume that there exists $\overline{t} \in ]\wtau_2, T]$ such that $p_2(\overline{t}) + x_2(\overline{t})  = 0$. Then $\overline{t}$ is a maximum point for the function $\Delta_2 \colon t \in  [\wtau_2, T ] \mapsto p_2(t) + x_2 (t) \in \R$. Thus 
\begin{equation}
\Delta_2 (\overline{t}) = 0, \qquad \dot \Delta_2 (\overline{t}) = 0, \qquad \ddot \Delta_2 (\overline{t}) \leq 0.
\end{equation}
Since $\ddot \Delta_2 (t) = \alpha \left( 2 - p_1 + \alpha\Delta_2(t) \right) \leq \ddot \Delta_2 (\overline{t})$ for any $t \in [ \wtau_2, T]$, and $\Delta_2 (\wtau_2) = 0$, we must have $\Delta_2 (t ) \equiv 0$. This implies $p_1 = 2$ and $x_2(t) \equiv \frac{-1}{\alpha}$, a contradiction.

}

 \medskip
 \noindent
 {\it Assumption~\ref{RS}.} First of all, we compute the expressions of the maximized Hamiltonians 
 \begin{gather*}
 H_1(p,x)=p_1 x_2 +p_2 (1-\alpha x_2) -x_2\qquad 
 H_2(p,x)=p_1 x_2 -\alpha p_2 x_2 \\ 
  H_3(p,x)=p_1 x_2 -p_2 (1+\alpha x_2) -x_2
 \end{gather*}
 In particular, we obtain
\begin{equation}
\big\{ H_1, H_2 \big\}=p_1-\alpha(x_2+p_2) \qquad
\big\{ H_2, H_3 \big\}=p_1+\alpha(x_2-p_2). 
\end{equation}

\noindent
Integrating the control system, we obtain that $x_2(\wtau_1)=\frac{1}{\alpha}\big(1-e^{-\alpha \wtau_1}\big)$, so that, using
the fact that $p_2(\wtau_1)=x_2(\wtau_1)$, we obtain 
\[
p_1(\wtau_1)-\alpha(x_2(\wtau_1)+p_2(\wtau_1))=\alpha x_2(\wtau_1) 
\frac{e^{-\alpha\wtau_2}(e^{-\alpha(\wtau_2-\wtau_1)}-e^{\alpha(\wtau_2-\wtau_1)}+2)}{e^{\alpha(\wtau_2-\wtau_1)}-1}.
\]
In particular, since $\wtau_2>\wtau_1$ and $x_2(\wtau_1)>0$, we just have to check the positivity of 
$e^{-\alpha(\wtau_2-\wtau_1)}-e^{\alpha(\wtau_2-\wtau_1)}+2$,
which is guaranteed whenever $e^{\alpha(\wtau_2-\wtau_1)}<1+\sqrt{2}$.
Using the expression for the switching times \eqref{ex wtau1}-\eqref{ex wtau2}, we can prove that this conditions holds for $T<T_{lim}$. 
Analogously, since $x_2(\wtau_2)=\frac{1}{\alpha}e^{-\alpha \wtau_2}(e^{\alpha \wtau_1}-1)$ and  $p_2(\wtau_2)=-x_2(\wtau_2)$ we obtain that 
\[
p_1(\wtau_2)+\alpha(x_2(\wtau_2)-p_2(\wtau_2))=\alpha x_2(\wtau_2)\Big(  
\frac{1+e^{2\alpha(\wtau_2-\wtau_1)}}{e^{\alpha(\wtau_2-\wtau_1)}-1}+2
\Big), \]
which is always positive.

 \medskip
 \noindent
 {\it Assumption~\ref{hp:secvar}.}
First of all, we compute the pull-back vector fields. Obviously, $g_1=h_1=f_0+f_1$.
To compute $g_2$ and $g_3$, we consider the functions 
\begin{gather*}
\varphi_2(x,t)=\exp(-th_1)_*h_2 \exp(t h_1)(x)\\ 
\varphi_3(x,t,s)=\exp(-th_1)_*\exp(-sh_2)_*h_3
 \circ\exp(sh_2)\circ\exp(t h_1)(x),
\end{gather*}
and notice that $g_2(x)=\varphi_2(x,\wtau_1)$ and $g_3(x)=\varphi_3(x,\wtau_1,\wtau_2-\wtau_1)$. The pull-back vector fields can be then computed developing $\varphi_2,\varphi_3$ in
powers of $t$ and $(t,s)$, respectively; we obtain
\[
g_2=f_0+ \frac{\eta(\wtau_1)}{\alpha}f_{01} \qquad
g_3=f_0-f_1+ \frac{\eta(\wtau_1)+\eta(\wtau_2)}{\alpha}f_{01},
\]
where $\eta(t)=1-e^{\alpha t}$ and $f_{01}=[f_0,f_1]$.

Let $(\delta x,\boldsymbol{\varepsilon})$ be an admissible variation contained in $V_0$; from $\delta x=(0,0)$ and the linear
independence of $f_1$ and $f_{01}$, we obtain that $(\ep_1,\ep_2,\ep_3)$ must satisfy the system
\[
 \begin{cases}
 \ep_1+\ep_2+\ep_3=0\\
 \ep_1-\ep_3=0\\
 \ep_2 \eta(\wtau_1)+\ep_3( \eta(\wtau_1)+\eta(\wtau_2))=0,
 \end{cases}
\]
so that, if $\wtau_1\neq \wtau_2$ (which is always verified for $T>T_{min}$), then $V_0$ is the trivial linear space, and the second variation is coercive  by definition.
\begin{remark}
For $T=T_{lim}$, Assumption~\ref{RS} is not satisfied, so we cannot apply our methods to prove optimality.
We believe that this is due to the fact that, whenever $T>T_{lim}$, the optimal control has the form bang-singular-zero-bang and, for bang-singular junctions, 
a necessary optimality condition is that the time-derivative of the difference of the maximized Hamiltonians along the two arcs is zero. By continuity, this conditions holds
also at the limit case $T=T_{lim}$.
\end{remark}
\section{Conclusions}\label{sec:final}
The paper contains a first attempt to extend the Hamiltonian methods  to problems where the cost has an $L^1$ growth with respect to the control and presents some 
non-smoothness issues as a function of the state.

 The family of costs under study is inspired by some examples from the literature, see 
 \cite{maurer-vossen,boizot}, 
 and we considered the case of a concatenation of arcs of the kind bang-zero-bang. 
 The result can be easily extended to the case, where more than three arcs are present, but in order to avoid a heavy notation and 
 still present all the difficulties due to the presence of a vanishing weight $\psi$, 
 we considered the case of two bang arcs only, while considering an arbitrary number of zeros for $\psi$.

The authors are currently considering the situation, where the extremal contains a singular arc, as the examples in \cite{boizot} show that such a situation may occur.

\appendix
\section*{Appendices}\label{appendix}
Here we give some hints for the computations appearing in Section \ref{sec:Ham}. 

First of all let us recall that for every Hamiltonian vector field $\vF{}$ and every Hamiltonian flow $\cK_t$, the following identity holds
for every fixed $t$:
\begin{equation}\label{eq:trasportoflusso}
\cK_{t*}\vF{} = \overrightarrow{F\circ{\cK_t}^{-1}}.
\end{equation}
\section{The Maximized Hamiltonian Flow: details on its construction} \label{app tau}
In Section~\ref{sec: cstr max Ham}, we saw that, for $t$ in a small left neighborhood of $\tau_1(\ell)$ and for any $\ell \in \mathcal{U}$, 
$\big(F_0+F_1 +\sigma_{0\, n_1+1}\psi(\pi(\ell))\big)\circ \mathcal{H}_t(\ell)$ is the maximized Hamiltonian.
Assumption~\ref{RS} implies that
the function 
\[\big(F_1(\ell)-|\psi(\pi\ell)|\big)\circ \exp \big((t- s_{1 n_1}(\ell)) \vH1^{\sigma_{0\,n_1+1}}\big) \circ\mathcal{H}_{s_{1n_1}(\ell)}(\ell)\]
is strictly decreasing with respect to $t$, for $t$ belonging to a neighborhood  of $\tau_1(\ell)$, so that 
$H_1^{\sigma_{0\,n_1+1}}$ cannot be the maximized Hamiltonian on $\exp \big((t- s_{1 n_1}(\ell)) \vH1^{\sigma_{0\,n_1+1}}\big) \circ\mathcal{H}_{s_{1n_1}(\ell)}(\ell)$ if $t>\tau_1(\ell)$. 
Indeed, for such $t$ the maximized Hamiltonian flow is
\[
\exp \big((t- \tau_1(\ell)) \vH2\big) \circ\mathcal{H}_{\tau_1(\ell)}(\ell).
\]
Along these curves, for $t>\tau_1(\ell)$, $H_2$ is the maximized Hamiltonian until 
the hypersurface $\{H_2-H_3^{\sigma_{21}}=0\}$ is reached. 
The time $\tau_2(\ell)$ for this to happen is, as $\tau_1(\ell)$, a smooth function of $\ell$, thanks to Assumption \ref{RS}. 
The proof is completely analogous to that of Proposition~\ref{prop:tau1}.
\begin{proposition}\label{prop:tau2}
Possibly shrinking $\cU$, for every $\ell \in \cU$ there exists a unique $\tau_2 = \tau_2(\ell)$ such that
$\tau_2(\lo) = \wh\tau_2$ and
\begin{equation*}
\left(H_3^{\sigma_{21}}-H_2 \right)\circ  \exp (\tau_2(\ell) - \tau_1(\ell))\vH{2}\circ \mathcal{H}_{\tau_1(\ell)}(\ell) = 0.
\end{equation*}
Moreover, the differential of $\tau_2(\ell)$ at $\well_0$ is given by 
\begin{align}
\label{eq:dtau2}
\scal{d\tau_2(\lo)}{\dl} &=  \dfrac{1}{\dueforma{\vH{2}}{\vH{3}^{\sigma_{21}}}(\well_2)}
\bigg\{
- \, \dueforma{\cHref_{\wh\tau_2*}\dl}{(\vH{3}^{\sigma_{21}} - \vH{2})(\well_2)} 
 \\ 
&    + \scal{d\tau_1(\lo)}{\dl}\,
\dueforma{\cHref_{\wh\tau_2 *}\cHref_{\wh\tau_1 *}^{-1}
\left(
\vH{2} - \vH{1}^{\sigma_{0\,n_1+1}}
\right)(\well_1)
}{\vH{3}^{\sigma_{21}} - \vH{2}}(\well_2) \nonumber
\\
& 
+ 2\sum_{i=1}^{n_1} \sigma_{0i} \dfrac{ \lieder{\pi_*\dl}{\wh\psi_{\sss_{1i}}}\, \lieder{g_3 - g_2}{\wh\psi_{\sss_{1i}}}}
{\lieder{g_1}{\wh\psi_{\sss_{1i}}}} 
\bigg\}. \nonumber
\end{align}
\end{proposition}
We thus extend the maximized Hamiltonian flow to the interval $[0,\tau_2(\ell)]$ setting 
\[
\mathcal{H}_t(\ell)=\exp \big((t- \tau_1(\ell)) \vH2\big) \circ
\mathcal{H}_{\tau_1(\ell)}(\ell), \qquad t\in[\tau_1(\ell),\tau_2(\ell)]. 
\]
The remaining points of non-smoothness of the maximized Hamiltonian flow can be characterized in an analogous way, as the following proposition shows.
\begin{proposition}\label{prop:s2}
Possibly shrinking $\mathcal{U}$, for every $i=1,\ldots,n_3$ there exist a unique smooth function $s_{3i} \colon \mathcal{U} \to \mathbb{R}$ such that
$s_{3i}(\well_0) = \sss_{3i}$ and
\[
\Psi\circ 
\exp (s_{3i}(\ell)- s_{3 i}(\ell)) \vH3^{\sigma_{2 i}} \circ \cdots \circ 
\exp ((s_{31}(\ell)-\tau_2(\ell)) \vH3^{\sigma_{21}})\circ \mathcal{H}_{\tau_2(\ell)}(\ell) = 0 \quad \forall \ell\in \mathcal{U}.
\]
Moreover, the differential of $s_{3i}$ at $\well_0$ is given by
\begin{align*}
\langle d s_{3i} (\well_0),\delta \ell \rangle &= 
\frac{-1}{L_{g_3} \wh\psi_{\sss_{3i}}(\wbx_0)} \Big\{
L_{\pi_*\delta \ell} \wh\psi_{\sss_{3i}}(\wbx_0) -\langle d\tau_1(\well_0), \delta \ell \rangle \, L_{g_2-g_1} \wh\psi_{s_{3i}}(\wbx_0)\\
&-
\langle d\tau_2(\well_0), \delta \ell \rangle \, L_{g_3-g_2} \wh\psi_{s_{3i}}(\wbx_0)
\Big\},\qquad i=1,\ldots,n_3.
\end{align*}
\end{proposition} 
\section{Details on the Hamiltonian Form of the Second Variation} \label{app H}
Consider the Hamiltonian function $G_1\se$ defined in \eqref{G1}, and compute its associated Hamiltonian vector field  
\[
\begin{pmatrix}
- \liederder{\cdot}{g_1}{\left( \wh\beta + \displaystyle\int_{\wI_3} \abs{\wh\psi_s}ds\right)}
- \sigma_{0n_1}\lieder{\cdot}{\wh\psi_{\wtau_{1}}} 
-2\displaystyle\sum_{i=1}^{n_3}\sigma_{2i} 
\lieder{\cdot}{\wh\psi_{\wh s_{3i}}} \dfrac{ \lieder{g_1}{\wh\psi_{\wh s_{3i}}}
}
{
\lieder{g_3}{\wh\psi_{\wh s_{3i}}}
}  \\
g_1(\wbx_0) 
\end{pmatrix} .
\]
Applying the antisymplectic isomorphism $\iota$ to the vector here above we obtain
\begin{multline} \label{eq:vvG1}
 \liederder{\cdot}{g_1}{\left( \wh\beta + \displaystyle\int_{\wI_3} \abs{\wh\psi_s}ds\right)}
+ \sigma_{0n_1}\lieder{\cdot}{\wh\psi_{\wtau_{1}}} \\%
+ D^2\left( \wh\beta + \displaystyle\int_{\wI_1 \cup \wI_3} \abs{\wh\psi_s}ds\right)(\wbx_0)[g_1, \cdot]
- 2 \displaystyle\sum_{i=1}^{n_1}\sigma_{0i} 
\lieder{\cdot}{\wh\psi_{\wh s_{1i}}}  
\end{multline}
for the component in $T^*_{\wbx_0}M$, and $g_1(\wbx_0)$ for the component in $T_{\wbx_0}M$. \\
Adding and subtracting $D^2\alpha(\wbx_0)[g_1, \cdot]$ and noticing that $d \left(\alpha + \wh\beta  + \!\!\displaystyle\int\limits_{\wI_1 \cup \wI_3} \!\! \abs{\wh\psi_s}ds\right) (\wbx_0) = 0$ 
by PMP, we obtain that 
\[
 D^2\left(\alpha +  \wh\beta + \displaystyle\int_{\wI_1 \cup \wI_3} \abs{\wh\psi_s}ds\right)(\wbx_0)[g_1, \cdot]
 = \liederder{\cdot}{g_1}{\left(\alpha + \wh\beta + \displaystyle\int_{\wI_1 \cup \wI_3} \abs{\wh\psi_s}ds\right)}
\]
so that \eqref{eq:vvG1} reduces to
\begin{align}
& \liederder{\cdot}{g_1}{\left( -\alpha - \int_{\wI_1} \abs{\wh\psi_s}ds \right) }
+ \sigma_{0n_1}\lieder{\cdot}{\wh\psi_{\wtau_{1}}}
- 2 \displaystyle\sum_{i=1}^{n_1}\sigma_{0i} 
\lieder{\cdot}{\wh\psi_{\wh s_{1i}}}  =  \\%
 &= - d\alpha(\wbx_0) Dg_1(\wbx_0) (\cdot) - \sigma_{01}\lieder{\cdot}{\wh\psi_{0}}
= - d\alpha(\wbx_0) Dh_1(\wbx_0) (\cdot) - \sigma_{01}\lieder{\cdot}{\psi}.%
\end{align}
This proves that the Hamiltonian vector field associated with $G_1\se$ is $\iota^{-1}\vH{1}^{\sigma_{01}}$, i.e.~$\vGse{1}$.

Analogous computations show that applying the antisymplectic isomorphism $\iota$ to the vector field associated with $G_2\se$ we obtain
\begin{equation}
\label{eq:vvG2}
\begin{pmatrix}
-d\alpha(\wbx_0)Dg_2(\wbx_0) -
\liederder{\cdot}{g_2}{\displaystyle\int_{\wI_1} \abs{\wh\psi_s}ds}
- \displaystyle\sum_{i=1}^{n_1}\sigma_{0i}
\dfrac{\lieder{\cdot}{\wh\psi_{s_{1i}}} \lieder{g_2}{\wh\psi_{s_{1i}}}}
{\lieder{g_1}{\wh\psi_{s_{1i}}}}
\\
g_2(\wbx_0) 
\end{pmatrix} .
\end{equation}
Let us choose a coordinate system $(p, q)$ in $T^*M$ and let
$(p_0, q_0)$ be some point.
Since 
\begin{equation}
\label{eq:flussocoord}
\cH_{\wtau_1}\begin{pmatrix}
p_0\\q_0
\end{pmatrix}
= 
\begin{pmatrix}
\left( p_0 + \displaystyle\int_{\wI_1} d\abs{\psi_s}ds\right)\wh S^{-1}_{\wtau_1 *} \\
 \wh S_{\wtau_1}(q_0) 
\end{pmatrix},
\end{equation}
applying \eqref{eq:trasportoflusso} to $\vH{2}\circ\cH_{\wtau_1}$ we get the proof of \eqref{eq:Gsecondo} for $i=2$. The proof for $i=3$ follows the same lines.

We now want to compute the terms of the kind $G_i\se(\vGse{j})$, for $i,j=1,2,3$, as they are needed to prove equation \eqref{eq:JGG}.
It is sufficient to compute only three of these terms, since, as we will see, they are antisymmetric 
in $i,j$.

By straightforward computations it is easy to see that 
\[
G_2\se(\vGse{1})=
L_{[g_1,g_2]} \Big(\wh \beta+\int_{\wI_3} \big|\wh \psi_s\big| ds\Big)(\wbx_0) -
a_0 (-1)^{n_1} L_{g_2} \wh \psi_{\wtau_1} (\wbx_0)
\]
Moreover 
\[
\dueforma{\vH{1}^{\sigma_{0 \,n_1+1}}}{\vH{2}}(\well_1) = -u_1 \langle \well_1, [f_0,f_1]\rangle - \sigma_{0 n_1} \lieder{g_2}{\wh\psi_{\wtau_1}}.
\]
Applying \eqref{eq:flussocoord} and using the fact that 
$
d \left(\alpha + \! \displaystyle\int_{\wI_1} \abs{\wh\psi_s}ds\right) (\wbx_0) =  - d \left(\wh\beta + \! \displaystyle\int_{\wI_3} \abs{\wh\psi_s}ds\right) (\wbx_0),
$
(again by PMP) we get  the following equality:
\begin{equation}
G_2\se(\vGse{1})=
- \dueforma{\vH1^{\sigma_{0\,n_1+1}}}{\vH2}(\well_1)
  \label{G2G1}
  \end{equation}
  
Analogous computations prove the following identities:

\begin{align}
G_3\se(\vGse{2})
&=
L_{[g_2,g_3]} \Big(\wh \beta+\int_{\wI_3} \big|\wh \psi_s\big| ds\Big)(\wbx_0) +
a_2 L_{g_2} \wh \psi_{\wtau_2} (\wbx_0)\nonumber\\
&=- \dueforma{\vH2}{\vH3^{\sigma_{21}}}(\well_2)
 \label{G3G2} \\
 G_3\se(\vGse{1})
&=
L_{[g_1,g_3]} \Big(\wh \beta+\int_{\wI_3} \big|\wh \psi_s\big| ds\Big)(\wbx_0) -
a_0(-1)^{n_1} L_{g_3} \wh \psi_{\wtau_1} (\wbx_0) + a_2 L_{g_1} \wh \psi_{\wtau_2} (\wbx_0)\nonumber\\
&=- \dueforma{\widehat{\mathcal{H}}_{\wtau_2*}\widehat{\mathcal{H}}_{\wtau_1*}^{-1}\vH1^{\sigma_{1\, n_1+1}}}{\vH3^{\sigma_{21}}}(\well_2)
 \label{G3G1} 
\end{align} 

\medskip
We end this section adding some useful formulas, that can be recovered by tedious but simple computations. The one-form
$\omega_0(\delta x,\cdot) $ such that
 $\iota^{-1}d\alpha_*\dx = \left(\omega_0(\delta x,\cdot), \dx \right)$ is equal to 
\begin{align*}
\omega_0(\delta x,\cdot) &=  -D^2\Big( \alpha+\wh \beta +\int_{\wI_1\cup \wI_3} \big|\wh \psi_s\big| ds\Big) (\wbx_0)[\delta x,\cdot] \\
&- 2 a_0
\sum_{i=1}^{n_1} (-1)^i \dfrac{ 
\lieder{\delta x}{\wh\psi_{\wh s_{1i}}} \lieder{\cdot}{\wh\psi_{\wh s_{1i}}}}
{\lieder{g_1}{\wh\psi_{\wh s_{1i}}}}-
2 a_2 \sum_{i=1}^{n_3} (-1)^i \dfrac{ 
\lieder{\delta x}{\wh\psi_{\wh s_{3i}}} \lieder{\cdot}{\wh\psi_{\wh s_{3i}}}}
{\lieder{g_3}{\wh\psi_{\wh s_{3i}}}}. 
\end{align*}
Thanks to this, we can prove that
\begin{align} 
 \scal{d\tau_1(\lo)}{d\alpha_*\dx} &= \dfrac{-1}{G_2\se(\vGse{1})}
\left(
G_2\se  - G_1\se
\right)(\omega_0(\dx), \dx) \label{G2-G1}\\ 
 \scal{d\tau_2(\lo)}{d\alpha_*\dx} &= \dfrac{-1}{G_3\se(\vGse{2})}
\left(
G_3\se  - G_2\se
\right)\left( (\omega_0(\dx), dx) 
- \scal{d\tau_1(\lo)}{d\alpha_*\dx}(\vGse{2} - \vGse{1})
\right)
\end{align}
and we can compute the expression of the bilinear symmetric form associated with
\begin{align*}
\mathfrak{J}[\delta e,\delta f]&=
\dfrac{1}{2}\Big(
\gamma\se[\dx, \dy]+\scal{\omega_0(\dx)}{\dy}  + \eta_1 G\se_1(\omega_0(\dx), \dx)  
+\eta_2 G\se_2\big( (\omega_0(\dx), \dx) + \ep_1 \vGse{1} \big)\\
&+\eta_3 G\se_3\big( (\omega_0(\dx), \dx) + \ep_1 \vGse{1} + \ep_2 \vGse{2}\big)
\Big) ,
\end{align*}
where 
\[
\gamma\se[\dx, \dy] := \scal{-\omega_0(\dx)}{\dy}
\]
(see \cite{PoggTo} for more details).

\begin{acknowledgements}
This work was supported by  the projects ``Dinamiche non autonome, sistemi Hamiltoniani e teoria del controllo'' (GNAMPA 2015), 
``Dinamica topologica, sistemi Hamiltoniani
e teoria del controllo'' (GNAMPA 2016) and by UTLN - Appel \`a projet
``Chercheurs invit\'es''.
\end{acknowledgements}

%

\end{document}